\documentclass[reqno,10pt,centertags,draft]{amsart}
\usepackage{amsmath,amsthm,amscd,amssymb,latexsym,upref,stmaryrd}
\usepackage{bbm}
\usepackage{mathtools}

\newcommand{\R}{{\mathbb R}}
\newcommand{\N}{{\mathbb N}}

\newcommand{\C}{{\mathbb C}}

\newcommand{\bbK}{{\mathbb{K}}}

\newcommand{\cF}{{\mathcal F}}
\newcommand{\cG}{{\mathcal G}}

\newcommand{\cP}{{\mathcal P}}

\DeclareMathOperator{\diag}{diag}

\newcommand {\ang} {\measuredangle}

\newcommand{\beq}{\begin{equation}}
\newcommand{\enq}{\end{equation}}

\newcommand{\lb}{\label}

\renewcommand{\ge}{\geqslant}
\renewcommand{\le}{\leqslant}

\numberwithin{equation}{section}

\allowdisplaybreaks \numberwithin{equation}{section}

\newtheorem{theorem}{Theorem}[section]

\newtheorem{proposition}[theorem]{Proposition}
\newtheorem{lemma}[theorem]{Lemma}
\newtheorem{corollary}[theorem]{Corollary}

\newtheorem{definition}[theorem]{Definition}

\theoremstyle{remark}
\newtheorem{remark}[theorem]{Remark}

\newtheorem{example}[theorem]{Example}

\begin{document}

\numberwithin{equation}{section}
\allowdisplaybreaks

\title{On a generalized notion of metrics}
  \thanks{Supported by the Deutsche Forschungsgemeinschaft (DFG, German Research
 Foundation) – SFB 1283/2 2021 – 317210226, Bielefeld University}



  \author{Wolf-J\"urgen Beyn}
  \address{Department of Mathematics, Bielefeld University\\
    P.O. Box 100131, D-33501 Bielefeld}
  \email{beyn@math.uni-bielefeld.de}


\vspace{-3mm}

\begin{abstract}
  In these notes we generalize the notion of a (pseudo) metric measuring the distance of  two
points,  to a (pseudo) $n$-metric which assigns a value to a tuple 
of $n \ge 2$ points. We present two principles  of constructing pseudo
$n$-metrics. The first one uses the Vandermonde determinant
while the second one  uses exterior products and is related to
the volume of the simplex spanned by the given points. We show that the second class
of examples induces pseudo $n$-metrics  on the unit sphere of a Hilbert space and on matrix manifolds such as the Stiefel
and the Grassmann manifold. Further, we construct a pseudo $n$-metric on
hypergraphs and discuss the problem of generalizing the Hausdorff metric
for closed sets to a pseudo $n$-metric.
\end{abstract}
  
\maketitle
\vspace{-3mm}

\vspace{1cm}

 \vspace{-3mm}
 
\section{Introduction}
\label{intro}
The classical notion of a (pseudo) metric or distance on a set $X$ is a map which
assigns to any two points in $X$ a real value such that $d$ is (semi-)definite,
symmetric, and satisfies the triangle inequality. Since the subject is vast
we just mention \cite{Ma22} as an elementary and \cite{Bu01} as an advanced monograph.
The purpose of this contribution is to investigate an $n$-dimensional
($n \ge 2$) generalization of this notion.
To be precise, we call a map
\begin{align*} d:X^n=\prod_{i=1}^n X \to \R
\end{align*}
a {\it pseudo $n$-metric on $X$} if it has the following properties:
  \begin{itemize}
    \item[(i)] (Semidefiniteness) 
    If $x=(x_1,\ldots,x_n)\in X^n$ satisfies $x_i=x_j$ for some $i\neq j$, then
    $d(x_1,\ldots,x_n)=0$.
  \item[(ii)] (Symmetry) 
    For all $x=(x_1,\ldots,x_n)\in X^n$ and all 
      $\pi\in \mathcal{P}_n$
      \begin{equation} \label{perm}
        d(x_{\pi(1)},\ldots,x_{\pi(n)})= d(x_1,\ldots,x_n).
      \end{equation}
      Here $\mathcal{P}_n$ denotes the group of permutations of $\{1,\ldots,n\}$.
    \item[(iii)] (Simplicial inequality)
      For all $x=(x_1,\ldots,x_n)\in X^n$ and $y \in X$
      \begin{equation} \label{tri}
        d(x_1,\ldots,x_n) \le \sum_{i=1}^n d(x_1,\ldots,x_{i-1},y,x_{i+1},\ldots,x_n).
      \end{equation}
      \end{itemize}
Clearly, a pseudo $2$-metric agrees with the notion of an ordinary
pseudo metric.
If we sharpen condition (i) to definiteness, i.e. \
 $ d(x_1,\ldots,x_n) = 0$ implies $x_i=x_j$ for some $i \neq j$,
then we call $d$ an $n$-metric.

We have two basic principles of constructing pseudo $n$-metrics. The first
one is based on the Vandermonde determinant (Section \ref{s2}). It leads
to a definite $n$-metric in the complex plane (or equivalently in $\R^2$)
and can be used to define pseudo $n$-metrics in spaces of complex-valued
functions (Theorem \ref{Cprop}, Examples \ref{ex1} and \ref{ex2}).
Although the straightforward generalization of the Vandermonde determinant to a product
of norms does not work, we present in Section \ref{sec6.0} an alternative generalization
via the norm of symmetric multilinear forms.

The second class of  pseudo $n$-metrics occurs in Eudlidean
vector spaces when $d(x_1,\ldots,x_n)$ is taken as the volume of
the simplex $\mathcal{S}(x_1,\ldots,x_n)$  spanned by $x_1,\ldots,x_n$, or (up to a factor of $(n-1)!$) as the
volume of the parallelepiped spanned by $x_2-x_1,\ldots,x_n-x_1$; see Theorem \ref{propcons}. While semidefiniteness and symmetry are obvious in this case,
the simplicial inequality reflects the fact that 
 $\mathcal{S}(x_1,\ldots,x_n)$ is covered by the simplices
$\mathcal{S}(x_1,\ldots,x_{i-1},y,x_{i+1},\ldots,x_n)$ for  $i=1,\ldots,n$, with
equality in case $y\in \mathcal{S}(x_1,\ldots,x_n)$. For $n \ge 3$
this pseudo $n$-metric is not definite but becomes zero when
the spanning vectors $x_2-x_1,\ldots,x_n-x_1$ are linearly dependent.
Instead of working directly with simplices we
present in Section \ref{s:lin} a more general approach which uses
pseudo $n$-norms and exterior calculus.

In differential geometry, the metric tensor  of a Riemannian manifold is
defined as a $2$-tensor on the tangent bundle. In principle, this is a local
definition using charts, it  does not  allow us to directly compare arbitrary
points on the manifold. This is then achieved by the Riemannian
distance  (or the metric in the sense of this paper) which employs
geodesics, i.e. \ paths of shortest length
connecting the two points, or more general metric structures;
see e.g. \cite[Ch.2,5]{Bu01}.

In \cite{Schullerwohlfarth06}, \cite{PSW09} the authors consider
  area-metrics on Riemannian manifolds which are metric $4$-tensors and which
  are more general than the canonical area metric induced by the metric
  $2$-tensor. Using this concept, some  interesting implications for
  gravity dynamics and electrodynamics are discussed.
  However, there seems to be no attempt to define a pseudo $3$-metric in our
  sense, i.e.\
  an area metric which is defined for any triple of points on the manifold
  and which satisfies the three properties above.
    Other generalizations of the metric notion, such as the 
  polymetrics in \cite{Mi98} or the multi-metrics in \cite{DR21}
  aim at spaces where several $2$-metrics are given, sometimes with
  different multiplicities.

  In Section \ref{sec4} we construct a specific pseudo $n$-metric on the unit
  ball of a Hilbert space, and we derive from this  suitable
  pseudo $n$-metrics on manifolds of matrices such as the Stiefel and
  the Grassmann manifold (see \cite{BZ21} for a recent overview of geometric
  and computational aspects).
  In Section \ref{sec5} we turn to metrics on discrete spaces and show
   how to extend the standard shortest path distance
  in an undirected graph  to a pseudo $n$-metric on
  a connected hypergraph; see \cite{Br13} for some general theory.
  
  Finally, in Section \ref{sec6} we discuss several externsions resp. open problems raised by the first sections.
  The first one deals with the generalization of the pseudo $n$-metric derived from the Vandermonde determinant.
  We show that the product of norms does not work in dimensions $\ge 3$, but that it is possible
  to construct pseudo $n$-metrics from the norm of a symmetric multilinear map which is
  of degree $\frac{1}{2}n(n-1)$.
  The second one is concerned with the search for a
  pseudo $n$-metric on the Grassmann manifold which is a canonical extension
  of the classical $2$-metric (\cite[Ch.6.4]{GvL2013}). The last
  one concerns a generalization of the Hausdorff metric for closed sets,
  where we show that an obvious definition fails to produce a pseudo $n$-metric.

  Due to its various areas of occurrence, we think that the axiomatic
  notion of a pseudo $n$-metric deserves to be studied in its own right.


\section{Basic properties and the Vandermonde example}\lb{s2}

In this secton we collect some basic constructions and properties of pseudo $n$-metrics. Our main example is based on the Vandermonde determinant which
leads to a pseudo $n$-metric that is homogeneous of degree $\frac{1}{2}n(n-1)$.
In the following, a mapping $d:X^n \to \R$ is called nonnegative if $d(x)\ge0$
holds for all $x\in X^n$. 
\begin{lemma} \label{lempos}
  \begin{itemize}
    \item[(i)]
        A pseudo $n$-metric on $X$ is nonnegative.
      \item[(ii)]  If the mapping $d:X^n \to \R$ is nonnegative, semidefinite,
        symmetric, and satisfies \eqref{tri} for all pairwise different elements $y,x_1,\ldots,x_n$,
     then $d$ is a  pseudo $n$-metric.
  \end{itemize}
  \end{lemma}
\begin{proof}
   {\it (i):} \\ 
   We apply  \eqref{tri} with $y=x_1$ and obtain from the semidefiniteness
   and the symmetry
  \begin{align*}
    0& = d(x_2,x_2,x_3,\ldots,x_n) \le
    d(x_1,x_2,x_3,\ldots,x_n)+ d(x_2,x_1,x_3,\ldots,x_n) \\
    & + \sum_{i=3}^n d(x_2,x_2,x_3,\ldots,x_{i-1},x_1,x_{i+1},\ldots,x_n) 
     = 2 d(x_1,x_2,\ldots,x_n).
  \end{align*}
  {\it (ii):}\\
  In case $x_i=x_j$ for some $i \neq j$,  the inequality \eqref{tri} is obvious
  since $d$ is semidefinite and nonnegative.
 Further, if \ $y = x_j$ for some $j \in \{1,\ldots,n\}$, then
    the semidefiniteness implies
  \begin{align*}
    d(x_1,\ldots,x_n)& =d(x_1,\ldots,x_{j-1},y,x_{j+1}, \ldots, x_n)\\
    & =
      \sum_{i=1}^n d(x_1,\ldots,x_{i-1},y,x_{i+1},\ldots,x_n).
  \end{align*}
\flushright \end{proof}

\begin{proposition} \label{prop1.1}
  For a pseudo $n$-metric $d$ on a set $X$ the following holds:
  \begin{itemize}
  \item[(i)] $d_X$ defines a pseudo $n$-metric on any subset $Y\subset X$.
  \item [(ii)] If $d_Y$ is a pseudo $n$-metric an a set $Y$, then
    \begin{equation*} \label{eqprod}
      d_{X \times Y}\Big( \begin{pmatrix} x_1 \\ y_1 \end{pmatrix}, \ldots,
      \begin{pmatrix} x_n \\ y_n \end{pmatrix} \Big) :=
      \Big\| \begin{pmatrix} d_X(x_1,\ldots,x_n) \\ d_Y(y_1, \ldots,y_n) \end{pmatrix} \Big\|
    \end{equation*}
    defines a pseudo $n$-metric on $X \times Y$ for any monotone norm $\|\cdot\|$
    in $\R^2$.
    \item[(iii)] Let $\cF \subseteq X^Y$ be a subset of functions from some
set $Y$ to $X$. Further assume that a normed space
$Z \subseteq \R^{Y}$ is given with a monotone norm $\|\cdot \|_Z$ and
the following property:
\begin{align} \label{Fprop}
  f_1,\ldots,f_n \in \cF \Longrightarrow d(f_1(\cdot),\ldots,f_n(\cdot))
  \in Z.
\end{align}
Then a pseudo $n$-metric on $\cF$ is defined by
\begin{equation*} \label{pseudoF}
  d_{\cF}(f_1,\ldots,f_n)  = \| d(f_1(\cdot),\ldots,f_n(\cdot)) \|_Z.
\end{equation*}
 \end{itemize}
\end{proposition}
\begin{proof}  Since most assertions are obvious we consider only the
  simplicial inequality in cases (ii) and (iii).
  Recall that a norm $\|\cdot\|_Z$
  on a linear space $Z$ of real-valued functions $f:Y \to \R$
  is called monotone (see e.g. \cite[Ch.5.4]{HJ2013}) iff for all $f_1,f_2 \in Z$
  \begin{align*}
    |f_1(y)| \le |f_2(y)| \; \forall y \in Y \Longrightarrow
    \|f_1\|_Z \le \|f_2\|_Z.
  \end{align*}
  In case (ii) one takes norms in the following vector inequality for $\xi\in X$,
  $\eta \in Y$
  \begin{align*}
    \begin{pmatrix} d_X(x_1,\ldots,x_n) \\ d_Y(y_1, \ldots,y_n) \end{pmatrix}
    \le \sum_{i=1}^n \begin{pmatrix} d_X(x_1,\ldots,x_{i-1},\xi,x_{i+1},\ldots,x_n) \\ d_Y(y_1,\ldots,y_{i-1},\eta,y_{i+1},\ldots,y_n) \end{pmatrix}.
\end{align*}
The proof of (iii) is analogous.
For $f_1,\ldots,f_n,g \in \cF$ observe the pointwise inequality
\begin{align*}
  d(f_1(y),\ldots,f_n(y)) & \le \sum_{i=1}^n d(f_1(y),\ldots,f_{i-1}(y),g(y),
  f_{i+1}(y),\ldots f_n(y)), \quad \forall y \in Y,
\end{align*}
which by the monotonicity of the norm implies
\begin{align*}
  d_{\cF}(f_1,\ldots,f_n) & \le \sum_{i=1}^n \| d(f_1(\cdot),\ldots,f_{i-1}(\cdot),g(\cdot),f_{i+1}(\cdot),\ldots,f_n(\cdot)) \|_Z \\
  & = \sum_{i=1}^n d_{\cF}(f_1,\ldots,f_{i-1},g,f_{i+1},\ldots, f_n).
\end{align*}
\flushright \end{proof}
Our  basic example of this section employs the well-known Vandermonde determinant associated with numbers $z_j \in \C$, $j=1,\ldots,n$ (see \cite[Ch.0.9]{HJ2013}):
\begin{equation} \label{CdefVandermond}
  V(z_1,\ldots,z_n) = \det \begin{pmatrix}
    z_1^{n-1} & \cdots & z_n^{n-1} \\
    \vdots & \cdots & \vdots \\
    1 & \cdots & 1 \end{pmatrix}=
  \prod_{1 \le i < j \le n} (z_i-z_j).
\end{equation}
\begin{theorem}(Vandermonde $n$-metric) \label{Cprop}
    The expression
    \begin{equation} \label{Cprodn}
      d_V(z_1,\ldots,z_n) = \prod_{1\le i < j \le n}|z_i-z_j|
      =|V(z_1,\ldots,z_n)|
    \end{equation}
    defines an $n$-metric in $\C$ and by restriction also in $\R$.
         \end{theorem}
\begin{proof} According to Proposition \ref{prop1.1}(i) it suffices to consider \eqref{Cprodn}
  in $\C$.
  By the definition \eqref{Cprodn} it is clear that $d_V$ is nonegative, semidefinite, and even definite. The symmmetry follows from
  \begin{align*}
    d_V(z_1,\ldots,z_n)^2 = \prod_{i \neq j}|z_i-z_j|= \prod_{\pi(i)\neq \pi(j)}
    |z_{\pi(i)}-z_{\pi(j)}|= d_V(z_{\pi(1)},\ldots,z_{\pi(n)})^2.
    \end{align*}
  By Lemma \ref{lempos}(ii) it is enough to show the simplicial inequality for pairwise different numbers $z_1,\ldots,z_n,y$. Then there is a unique
    solution $(a_1,\ldots,a_n)^{\top}\in \C^n$ to the linear system
    \begin{align*}
\begin{pmatrix}
    z_1^{n-1} & \cdots & z_n^{n-1} \\
    \vdots & \cdots & \vdots \\
    1 & \cdots & 1 \end{pmatrix} \begin{pmatrix}a_1 \\ \vdots \\ a_n
\end{pmatrix} = \begin{pmatrix} y^{n-1} \\ \vdots \\ 1 \end{pmatrix}.
    \end{align*}
    The solution is  given by Cramer's rule as follows
    \begin{align*}
  a_i = \frac{V(z_1,\ldots,z_{i-1},y,z_{i+1},\ldots,z_n)}{V(z_1,\ldots,z_n)},
  \quad i=1,\ldots,n,
    \end{align*}
    with the Vandermonde determinant from \eqref{CdefVandermond}.
    The proof is completed by the following estimate
    \begin{equation} \label{Vest}
    \begin{aligned}
     & d_V(z_1,\ldots,z_n) = |V(z_1,\ldots,z_n)\sum_{i=1}^n a_i|
       = |\sum_{i=1}^n V(z_1,\ldots,z_{i-1},y,z_{i+1},\ldots,n)|\\
      & \quad \le \sum_{i=1}^n |V(z_1,\ldots,z_{i-1},y,z_{i+1},\ldots,n)|
      = \sum_{i=1}^n d_V(z_1,\ldots,z_{i-1},y,z_{i+1},\ldots,n).
    \end{aligned}
    \end{equation}
\flushright \end{proof}
\begin{remark}
  It is remarkable that the proof yields the following more general
  inequality
  \begin{align*}
    |y|^k d_V(z_1,\ldots,z_n) \le \sum_{i=1}^n |z_i|^k d_V(z_1,\ldots,z_{i-1},y,z_{i+1},
    \ldots,z_n), \quad k=0,\ldots,n-1.
  \end{align*}
  Let us also note that a straightforward generalization of \eqref{Cprodn}
  to a product of norms fails in higher dimensional spaces; see Section
  \ref{sec6.0}. However, in Section \ref{sec6.0} we present an alternative generalization of the
  Vandermonde pseudo $n$-metrics.
\end{remark}
In case $n=3$ we investigate when the simplicial inequality holds with equality.
    \begin{proposition} \label{corcomplex}
      For the $3$-metric
    \begin{equation*} \label{Cprodn=3}
      d_V(z_1,z_2,z_3) = |z_1-z_2||z_2-z_3||z_1 -z_3|, \quad z_1,z_2,z_3\in \C
    \end{equation*}
   the equality
    \begin{align} \label{Csimpequality}
      d_V(z_1,z_2,z_3) = d_V(y,z_2,z_3)+d_V(z_1,y,z_3)+d_V(z_1,z_2,y)
    \end{align}
    holds with pairwise different numbers $y, z_1,z_2,z_3 \in \C$ if and only if
    the 
    quadruple $(y,z_1,z_2,z_3)$ belongs (up to a shift and a multiplication
    by a complex number and up to a permutation of $z_1,z_2,z_3$) to the following two parameter family $(q,s>0)$:
    \begin{equation} \label{Ceqfamily}
     y=0, \quad z_1=1, \quad z_2=\frac{1}{s}\Big(-1 + i \sqrt{q(1+s)}\Big),
      \quad z_3=\frac{1}{s}\Big(-1 - i \sqrt{q^{-1}(1+s)}\Big).
    \end{equation}
    Further, the equality \eqref{Csimpequality} holds with $|z_1-y|=|z_2-y|=
    |z_3-y|$  if and only if the triangle $z_1,z_2,z_3$ is equilateral. 
    \end{proposition}
    \begin{proof}
     From \eqref{Vest} we find that equality 
    holds in \eqref{tri} if and only if
    \begin{equation} \label{CeqV}
      \big| \sum_{i=1}^n V(z_1,\ldots,z_{i-1},y,z_{i+1},\ldots,z_n) \big|
      = \sum_{i=1}^n |V(z_1,\ldots,z_{i-1},y,z_{i+1},\ldots,z_n)|.
    \end{equation}
    Since the numbers are distinct we may shift $y$ to zero and multiply
    by a complex number such that $z_1=1$.
   Recall that the equality $|z_1|+|z_2|=|z_1+z_2|$ holds for $z_1,z_2\in \C$
    if and only if either $z_1=cz_2$ or $z_2=cz_1$ for some $c \ge 0$.
    More generally, the equality $|\sum_{i=1}^n z_i|=\sum_{i=1}^n |z_i|$
    holds if and only if there exists an index $j \in \{1,\ldots,n\}$
    and real numbers $c_i\ge 0$ such that $z_i = c_i z_j$ for $i=1,\ldots,n$.
    Moreover, if $z_i \neq 0$ for all $i=1,\ldots,n$ then the latter
    property holds for any $j\in\{1,\ldots,n\}$ with numbers $c_i>0$.
    We apply this to \eqref{CeqV} with $n=3$ after normalizing $y=0$
    and $z_1=1$. Thus the equality \eqref{Csimpequality} holds if and only
    if there are numbers $c_2,c_3>0$ such that
    \begin{equation*} \label{Csyst1}
      \begin{aligned}
        V(1,0,z_3)& = (-z_3)(1-z_3) = c_2 V(0,z_2,z_3)=c_2(-z_2)(-z_3)(z_2-z_3),\\     V(1,z_2,0)& = (1-z_2)z_2= c_3V(0,z_2,z_3)= c_3 (-z_2)(-z_3)(z_2-z_3).
      \end{aligned}
    \end{equation*}
    Since $z_2,z_3 \neq 0$ this is equivalent to the system
    \begin{equation*} \label{Csyst2}
      \begin{aligned}
        1-z_3& = - c_2z_2(z_2-z_3), \\
        1-z_2& = c_3 z_3(z_2-z_3).
      \end{aligned}
    \end{equation*}
    Subtract the second from the first equation and use $z_2 \neq z_3$
   to find the equivalent system
    \begin{equation} \label{Csyst3}
      \begin{aligned}
        1-z_3& = - c_2z_2(z_2-z_3), \\
        1& = -c_2z_2 -c_3 z_3. 
      \end{aligned}
    \end{equation}
    With the last equation we eliminate $z_3=-c_3^{-1}(1+c_2z_2)$ from
    the first  to obtain the quadratic equation
    \begin{align*}
      c_2(c_2+c_3)z_2^2+ 2 c_2 z_2 +c_3 + 1 =0.
    \end{align*}
    The solutions to \eqref{Csyst3} are then given by
    \begin{equation*}
      \begin{aligned}
        z_2^{\pm}&=\frac{1}{c_2+c_3}\Big(-1 \pm i \big( \frac{c_3}{c_2}(1+c_2+c_3)\big)^{1/2}\Big),\\
        z_3^{\mp}&
        =\frac{1}{c_2+c_3}\Big(-1 \mp i \big( \frac{c_2}{c_3}(1+c_2+c_3)\big)^{1/2}\Big).
      \end{aligned}
    \end{equation*}
    Introducing the parameters $s=c_2+c_3 >0$ and $q=\frac{c_3}{c_2}>0$
    we may write this as
    \begin{equation*}
      z_2^{\pm}(q,s)=\frac{1}{s}\Big(-1 \pm i \sqrt{ q(1+s)}\Big),
      \quad
        z_3^{\mp}(q,s)
        =\frac{1}{s}\Big(-1 \mp i \sqrt{ q^{-1}(1+s)}\Big).
    \end{equation*}
    Finally, we observe the relation
    \begin{align*}
      (z_2^{-},z_3^+)(q^{-1},s)=(z_3^-,z_2^+)(q,s), \quad q,s>0,
    \end{align*}
    where the latter is a permutation of $(z_2^+,z_3^-)(q,s)$.
    Hence the family \eqref{Ceqfamily} covers all solutions up to permutation.

    For the normalized coordinates $y=0, z_1=1$ the condition $1=|z_2|=|z_3|$
    holds if and only if $s=2,q=1$. For these values the numbers
    $z_2=\frac{1}{2}(-1 + i \sqrt{3})$, $z_3=\frac{1}{2}(-1 - i \sqrt{3})$
      are the third roots of unity which belong to the equilateral case.
    \flushright \end{proof}
    \begin{remark} A simple consequence of Theorem \ref{Cprop} and
      Proposition \ref{corcomplex} for $y=0$, $|z_1|=|z_2|=|z_3|=1$ is the following elementary statement: The sides $a,b,c$ of a triangle
      with vertices on the unit circle satisfy the inequality
      \begin{align*}
        abc \le a+b + c.
        \end{align*}
      Equality holds if and only if the triangle is equilateral.
      Neither did we find a reference to this fact in the literature nor an elementary
      geometric proof.
    \end{remark}
    Combining Proposition \ref{prop1.1} and Theorem \ref{Cprop} leads to
    the following examples of pseudo $n$-metrics.
    \begin{example} \label{ex1}
      Let $\|\cdot\|$ be any monotone norm in $\R^k$. Then
      \begin{align*}
        d_k(x_1,\ldots,x_n)= \left\| \begin{pmatrix} d_V(x_{1,1},\ldots,x_{n,1})\\
          \vdots \\ d_V(x_{1,k},\ldots,x_{n,k}) \end{pmatrix} \right\|,
        \quad x_j = (x_{j,i})_{i=1}^k \in \R^k
      \end{align*}
      defines a pseudo $n$-metric in $\R^k$. Note that $d_k$, $k \ge 2$ is
      not definite although $d_V$ is.
    \end{example}
    \begin{example} \label{ex2}
      Let $(\Omega,\cG,\mu)$ be  a bounded measure space. We apply Proposition
      \ref{prop1.1}
      to  $Z=L^p(\Omega,\cG;\R)$ with the monotone norm
      $\|\cdot\|_{L^p}$  and to $\cF = L^r(\Omega,\cG;\R)$ with  $2r \ge n(n-1)p$.
      Then the expression
\begin{align*}
  d_{L^r}(f_1,\ldots,f_n)& = \Big( \int_{\Omega} \prod_{1\le i < j \le n}|f_i- f_j|^p
  \, \mathrm{d}\mu \Big)^{1/p}.
\end{align*}
defines a pseudo $n$-metric on $L^r(\Omega,\cG;\R)$. Note that $2r \ge n(n-1)p$ guarantees
that the property \eqref{Fprop} holds.
      \end{example}
\section{Examples in linear spaces}
\label{s:lin}
\subsection{Pseudo $n$-norms}
In  vector spaces it is natural to first define 
a pseudo $n$-norm and then define a pseudo $n$-metric by taking the
$n$-norm of differences.
\begin{definition} \label{defpseudonorm}
  Let $X$ be a vector space. A map $\|\cdot \|:X^n \to \R$ is called
  a pseudo $n$-norm if it has the following properties:
  \begin{itemize}
  \item[(i)] (Semidefiniteness) \\
    If $x_i $, $i\in \{1,\ldots,n\}$ are linearly dependent
    then $\|(x_1,\ldots,x_n)\|=0$.
  \item[(ii)] (Symmetry) \\
    For all $x=(x_1,\ldots,x_n) \in X^n$ and for all $\pi \in S_n$
    \begin{equation} \label{normperm}
      \|(x_1,\ldots,x_n)\|= \|(x_{\pi(1)},\ldots,x_{\pi(n)}) \|.
    \end{equation}
  \item[(iii)] (Positive homogeneity) \\
    For all $x=(x_1,\ldots,x_n) \in X^n$ and for all
    $\lambda \in  \R$
    \begin{equation} \label{normhom}
      \|(\lambda x_1,\ldots,x_n)\|=
      |\lambda| \, \|(x_{1},\ldots,x_{n}) \|.
    \end{equation}
  \item[(iv)] (Multi-sublinearity) \\For all $x=(x_1,\ldots,x_n) \in X^n$ and
    $y \in X$,
    \begin{equation} \label{trinorm}
      \|(x_1+y,\ldots,x_n)\| \le \|(x_1,x_2,\ldots,x_n)\| +
      \|(y,x_2,\ldots,x_n)\|.
    \end{equation}
  \end{itemize}
  A pseudo $n$-norm is called an $n$-norm if $\|(x_1,\ldots,x_n)\|=0$
  holds if and only if $x_i,i=1,\ldots,n$ are linearly dependent.
\end{definition}
Clearly, a pseudo $1$-norm is a semi-norm in the usual sense. Also note
that the positive homogeneity \eqref{normhom} and the sublinearity \eqref{trinorm} transfers from the first component to all other components via \eqref{normperm}.

Next consider a matrix $A\in \R^{n,n}$ and the induced
linear map $A_X:=A^{\top}\otimes I_X$ on $X^n$ given by
  \begin{equation} \label{indmat}
    \begin{aligned}
      A_X \begin{pmatrix} x_1 \\ \vdots \\ x_n \end{pmatrix}& =
      \begin{pmatrix} A_{11}I_X & \cdots & A_{1n} I_X \\
        \vdots & & \vdots \\
        A_{n1}I_X & \cdots & A_{nn} I_X \end{pmatrix}
      \begin{pmatrix} x_1 \\ \vdots \\ x_n \end{pmatrix} \in X^n \\
      ( A_X (x_1,\ldots,x_n))_j &= \sum_{i=1}^n A_{ij}x_i, \quad j=1,\ldots,n.
    \end{aligned}
      \end{equation}

  \begin{lemma}\label{lemdet} Let $\|\cdot\|$ be a pseudo $n$-norm on $X$
    and $A \in \R^{n,n}$.
  Then the map \eqref{indmat} satisfies  
  \begin{equation} \label{detrule}
    \| A_X(x_1,\ldots,x_n)\| = |\det(A)|\,  \|(x_1,\ldots,x_n)\|
    \quad \forall (x_1,\ldots,x_n) \in X^n.
  \end{equation}
\end{lemma}
\begin{remark} The result shows that the pseudo $n$-norm is a 
  volume form on $X^n$.
\end{remark}
\begin{proof} If $A$ is singular then the vectors $(A_X(x_1,\ldots,x_n))_j$,
  $j=1,\ldots,n$ are linearly dependent, hence formula \eqref{detrule}
  is trivially satisfied. Otherwise there exists a decomposition
  $A=P L D U$ where $P$ is  an $n\times n$ permutation matrix, $L\in \R^{n,n}$
  resp. $U$ 
  is a lower resp. upper triangular $n \times n$-matrix with ones on the diagonal and $D\in \R^{n,n}$ is diagonal.
  Since $A_X =U_X D_X L_X P_X$ it suffices to show that all $4$ types of
  matrices satisfy the formula \eqref{detrule}. For $P_X$ this follows from
  \eqref{normperm} and for $D_X$ from \eqref{normhom}.
  For $L_X$ note that $L_X(x_1,\ldots,x_n)$ is of the form
  \begin{align*}
    L_X(x_1,\ldots,x_n) = ( \star, x_{n-1}+L_{n,n-1}x_n,x_n).
    \end{align*}
  With \eqref{trinorm} and Definition \ref{defpseudonorm}(i) we obtain
  \begin{align*}
    \| ( \star, x_{n-1}+L_{n,n-1}x_n,x_n)\|& \le \|( \star, x_{n-1},x_n)\|
    +\|( \star, L_{n,n-1} x_n,x_n)\| \\
    &=\|( \star, x_{n-1},x_n)\| .
  \end{align*}
  Similarly,
  \begin{align*}
    \|( \star, x_{n-1},x_n)\|& \le 
    \| ( \star, x_{n-1}+L_{n,n-1}x_n,x_n)\|
    +\|( \star,- L_{n,n-1} x_n,x_n)\| \\
    &= \| ( \star, x_{n-1}+L_{n,n-1}x_n,x_n)\|,
  \end{align*}
  hence we have
  \begin{align*}
    \|( \star, x_{n-1},x_n)\|=\| ( \star, x_{n-1}+L_{n,n-1}x_n,x_n)\|.
  \end{align*}
  By induction one eliminates all terms $L_{i,j}x_i,i>j$ from the pseudo
  $n$-norm. Since $\det(L)=1$ this proves \eqref{detrule} for $L$.
  In a similar manner, one shows
  \begin{align*}
    \|R_X(x_1,\ldots,x_n)\|&= \|(x_1,x_2 + R_{12}x_1,\star)\|
    = \|(x_1,x_2,\star)\|
  \end{align*}
  and finds that \eqref{detrule} is also satisfied for the map $R_X$.
  
    \flushright \end{proof}

\begin{proposition} \label{normtometric}
  Let $\| \cdot\|$ be a pseudo $(n-1)$-norm on a vector space $X$. Then
  \begin{align*}
    d(x_1,\ldots,x_n)= \|(x_2-x_1,\ldots,x_n-x_1)\|
  \end{align*}
  defines a pseudo $n$-metric on $X$.
\end{proposition}
\begin{proof}
  Let $x_i=x_j$ for some $i,j \in \{1,\ldots,n\}$, $i \neq j$.
  Then the vectors $x_{\nu}-x_1,\nu=2,\ldots,n$
  are linearly dependent, hence $d(x_1,\ldots,x_n)=0$ follows
  from  condition (i) of  Definition \ref{defpseudonorm}.
  Since $\cP_n$ is generated by transpositions, it is enough to
  show \eqref{perm} when transposing $x_1$ with some $x_j$,  $j \in \{2,\ldots,n\}$. For this purpose consider the matrix with $-1$ in the $j$-th row
  \begin{align*}
    A= \begin{pmatrix} 1& 0 & \cdots & \cdots & 0 \\
     \vdots & \ddots & & &\vdots  \\
      -1& -1 & -1 & -1 & -1 \\
     \vdots  & & & \ddots & \vdots \\
      0 & \cdots & \cdots & 0 & 1
    \end{pmatrix} \in \R^{n-1,n-1},
  \end{align*}
  which satisfies $\det(A)=-1$ and
  \begin{align*}
    & A_X (x_2-x_1, \ldots, x_n-x_1)\\
    &= (x_2-x_j, \ldots, x_{j-1}-x_j,x_1-x_j,
    x_{j+1}-x_j, \ldots,x_n -x_j).
  \end{align*}
  Taking norms and using \eqref{normperm} and \eqref{detrule} leads to
  \begin{align*}
    d(x_1,\ldots,x_n)=d(x_j,x_1,x_2,\ldots,x_{j-1},x_{j+1},\ldots,x_n).
  \end{align*}
  For the simplicial inequality \eqref{tri} we use the multi-sublinearity
  \eqref{trinorm}
  \begin{align*}
    d(x_1,\ldots,x_n) & = \|(x_2-y+y-x_1,\ldots, x_n-y + y -x_1) \|\\
    &\le   \|(x_2-y,x_3-x_1,\ldots,x_n-x_1)\| +
    \|(y-x_1,x_3-x_1,\ldots,x_n -x_1)\|\\
    & \le  \|(x_2-y,x_3-x_1,\ldots,x_n-x_1)\| + \|(y-x_1,x_3-y,\ldots,x_n -y)\|.
  \end{align*}
  For the  last term we used \eqref{detrule} and
  \begin{align*}
    (y-x_1,x_3-y,\ldots,x_n - y)& = U_X (y-x_1,x_3-x_1,\ldots,x_n-x_1)
  \end{align*}
  for the upper triangular matrix
  \begin{align*}
     U= \begin{pmatrix} 1 & -1 & \cdots & -1 \\ 0 & 1 &0 & 0 \\
      \vdots & &  \ddots & \vdots \\
      0 &  & \cdots & 1 \end{pmatrix}.
    \end{align*}
  Thus we have proved for $j=2$ the following assertion
  \begin{equation} \label{indtri}
    \begin{aligned}
      d(x_1,\ldots,x_n) &\le \sum_{i=2}^j d(y,x_1,\ldots,x_{i-1},x_{i+1},\ldots,x_n) \\
     & + \|(x_2-y,\ldots,x_j-y,x_{j+1}-x_1,\ldots,x_n -x_1)\|.
    \end{aligned}
  \end{equation}
  In the induction step we estimate the last term further by splitting
  $x_{j+1}-x_1 =x_{j+1}-y + y -x_1$:
  \begin{align*}
    & \|(x_2-y,\ldots,x_j-y,x_{j+1}-x_1,\ldots,x_n -x_1)\|\\
    & \le
    \|(x_2-y,\ldots,x_j-y,x_{j+1}-y,x_{j+2}-x_1,\ldots,x_n -x_1)\|\\
    & +\|(x_2-y,\ldots,x_j-y,y-x_1,x_{j+2}-x_1,\ldots,x_n -x_1)\|.
  \end{align*}
  As in the first step we can replace all vectors $x_i-x_1$, $i \ge j+2$
  in the last term by $x_i-y$.
  This proves that \eqref{indtri} holds for $j+1$. For $j=n$
  equation \eqref{indtri} reads
  \begin{align*}
   d(x_1,\ldots,x_n) &\le \sum_{i=2}^n d(y,x_1,\ldots,x_{i-1},x_{i+1},\ldots,x_n) \\
   & + \|(x_2-y,\ldots,x_n -y)\|= \sum_{i=1}^n d(y,x_1,\ldots,x_{i-1},x_{i+1},\ldots,x_n).
   \end{align*}

  \flushright \end{proof}

\subsection{Construction via exterior products}

The previous results suggest to use exterior products in defining appropriate
pseudo $n$-norms. For this purpose recall the following calculus
for a separable reflexive Banach space $X$; see \cite[Ch.V]{T97}, \cite[Section 6]{AMR88}, \cite[Ch.3.2.3]{A1998}.
The linear space $\bigwedge^k(X)$ is the set of continuous alternating 
$k$-linear forms on the dual $X^{\star}$, using the identification of
$X$ and its bidual $X^{\star \star}$. $\bigwedge^kX$ is called the $k$-fold exterior
product of $X$. Elements of $\bigwedge^k(X)$ are exterior products
$x_1 \wedge \ldots \wedge x_k$ defined for $x_1,\ldots,x_k\in X$
by the dual pairing
\begin{align*}
  \langle x_1 \wedge \ldots \wedge x_k, (f_1,\ldots,f_k) \rangle
  = \det \left( \langle f_i , x_j \rangle_{i,j=1}^k \right)
\end{align*}
where $f_1,\ldots,f_k \in X^{\star}$ and $\langle \cdot, \cdot \rangle$
is the dual pairing of $X^{\star}$ and $X$.
As usual we write in the following
\begin{align*}
  \bigwedge_{i=1}^k x_i = x_1\wedge \ldots \wedge x_k.
\end{align*}
By linearity one can extend exterior products to sums
\begin{align*}
  \sum_{j=1}^J c_j (x_1^j\wedge \ldots \wedge x_k^j), \quad c_j \in \R,
  x_i^j \in X,i=1,\ldots,k, j=1,\ldots,J \in \N.
\end{align*}
Closing the linear hull with respect to the norm
\begin{align*}
  \|\Phi\|_{\wedge}= \sup\{ | \Phi(f_1,\ldots,f_k)|: f_j \in X^{\star}, \|f_j\|_{X^{\star}}=1,j=1,\ldots,k \}, \quad \Phi\in {\bigwedge} ^k(X)
\end{align*}
turns $\bigwedge^k(X)$ into a Banach space.

The following lemma is elementary (see \cite[Lemma 3.2.6]{A1998})
\begin{lemma} \label{lemext}
  Let $\bigwedge^k X$ be the $k$-fold exterior product of a separable reflexive Banach space
 space $X$ with $k\le \dim(X)$. Then the following holds
  \begin{itemize}
  \item[(i)] If $X$ is $m$-dimensional ($m\ge k$) and $e_1,\ldots,e_m$ is a basis
    of $X$ then
    \begin{align*}
      \{e_{i_1}\wedge \ldots \wedge e_{i_k}: 1 \le i_1 < i_2 \ldots < i_k \le m\}
    \end{align*}
    is a basis of $\bigwedge^k X$. In particular $\bigwedge^k X$
    has dimension $m \choose k$.
  \item[(ii)] One has $x_1\wedge \ldots \wedge x_k=0$ if and only if
    $x_1,\ldots,x_k$ are linearly dependent.
  \item[(iii)] If $X$ is a Hilbert space with inner product $\langle \cdot, \cdot \rangle$  then the bilinear and continuous extension of
    \begin{equation} \label{eqinpro}
      \langle x_1\wedge \ldots \wedge x_k,y_1\wedge \ldots \wedge y_k \rangle
      = \det\big( \langle x_i,y_j \rangle_{i,j=1}^k\big) 
    \end{equation}
    defines an inner product on the Hilbert space $\bigwedge^k X$ . In particular,
    the corresponding norm
    \begin{equation} \label{eqvol}
      \| x_1 \wedge \ldots \wedge x_k \|_{\wedge} =
      \big(\det( \langle x_i,x_j \rangle_{i,j=1}^k) \big)^{1/2}
    \end{equation}
    is the volume of the $k$-dimensional parallelepiped spanned by
    $x_1,\ldots,x_k$.  Further, the generalized Hadamard  inequality holds
    for $j=1,\ldots,k$
    \begin{equation} \label{eqhad}
      \|x_1 \wedge \ldots \wedge x_k\|_{\wedge} \le \|x_1 \wedge \ldots
      \wedge x_j\|_{\wedge}
      \|x_{j+1} \wedge \ldots \wedge x_k\|_{\wedge}.
    \end{equation}
    
  \end{itemize}
\end{lemma}

\begin{proposition} \label{propcons}
  \begin{itemize}
  \item[(i)] Let $\bigwedge^nX$ be the $n$-fold exterior product of a separable reflexive Banach space
    $X$ and let $\|\cdot\|$ be a norm in $\bigwedge^nX$. Then
    \begin{equation} \label{defnnorm}
      \| (x_1,\ldots,x_n)\|= \big\| \bigwedge_{i=1}^n x_i \big\|
    \end{equation}
    defines an $n$-norm on $X$.
  \item[(ii)]
      Let $\bigwedge^{n-1} X$ be the $(n-1)$-fold exterior product of a
  Banach space $X$ and let $\| \cdot\|$ be a norm in $\bigwedge^{n-1} X$.
  Then
  \begin{equation*} \label{eqdefnm}
    d(x_1,\ldots,x_n)=\big \| \bigwedge_{i=2}^{n}(x_i-x_1)\big\|, \quad
    x=(x_1,\ldots,x_n) \in X^n
  \end{equation*}
  defines a pseudo $n$-metric on $X$. One has $d(x_1,\ldots,x_n)=0$ if and
  only if $ x_i-x_1, i=2,\ldots,n$ are linearly dependent.
  \end{itemize}
\end{proposition}
\begin{proof}
  By Proposition \ref{normtometric} it suffices to show that
  \eqref{defnnorm} defines an $n$-norm on $X$. Condition (i) in Definition
  \ref{defpseudonorm} follows from Lemma \ref{lemext} (ii), and \eqref{normperm}
  is a consequence of the alternating property of the exterior product.
  Further, the positive homogeneity \eqref{normhom}
  follows from the homogeneity of the exterior product and the positive
  homogeneity of the norm in $\bigwedge^nX$. Finally, inequality
  \eqref{trinorm} is implied by taking norms of the multilinear relation
  \begin{align*}
    (x_1+y)\wedge \bigwedge_{i=2}^n x_i = y \wedge \bigwedge_{i=2}^n x_i
    +\bigwedge_{i=1}^n x_i.
  \end{align*}
  
  \flushright
  \end{proof}

\begin{example} \label{L2}
  Let $(H,\langle \cdot,\cdot,\rangle_H,\|\cdot \|_H)$ be a Hilbert space.
  Then a $2$-norm on $H$ is given by
  \begin{equation} \label{hilbert2}
    \begin{aligned}
    \|(u,v)\|& = \left[ \det \begin{pmatrix} \|u\|_H^2 & \langle u,v\rangle_H \\
        \langle u,v \rangle_H & \|v\|_H^2 \end{pmatrix} \right]^{1/2}\\
    & = \left[ \|u\|_H^2 \|v\|_H^2 - \langle u,v \rangle_H^2 \right]^{1/2}\\
    & = \|u\|_H \|v\|_H(1 - \cos^2 \ang(u,v))^{1/2}=
    \|u\|_H \|v\|_H | \sin \ang(u,v) |.
    \end{aligned}
  \end{equation}
        The corresponding pseudo $3$-metric on $H$ is 
  \begin{equation*} \label{hilbert3}
    \begin{aligned}
    d(u,v,w)&=\left[ \det \begin{pmatrix} \|v-u\|_H^2 & \langle v-u, w-u \rangle_H \\
      \langle v-u, w-u \rangle_H & \|w-u\|_H^2 \end{pmatrix}\right]^{1/2}\\
    & =\left[ \|v-u\|_H^2\|w-u\|_H^2- \langle v-u, w-u \rangle_H^2\right]^{1/2}\\
    &= \|v-u\|_H \|w-u\|_H | \sin \ang(v-u,w-u) |.
    \end{aligned}
  \end{equation*}
  \end{example}

\section{Some pseudo $n$-metrics on manifolds}
\label{sec4}
We use the results from the previous section to set up a pseudo $n$-metric
on the unit sphere of a Hilbert space.
\begin{proposition} \label{metricball}
  Let $(H,\langle \cdot,\cdot,\rangle_H,\|\cdot \|_H)$ be a Hilbert space 
   and consider the unit  sphere
  \begin{align*}
    S_H=\{ x \in H: \| x\|_H = 1 \}.
  \end{align*}
  Then the  $n$-norm defined by \eqref{defnnorm} and \eqref{eqvol}
  generates a pseudo $n$-metric for  $x_i \in S_H$, $i=1,\ldots,n$ as
  follows
  \begin{equation*} \label{nmetricball}
    d(x_1,\ldots,x_n)= \big\|\bigwedge_{i=1}^n x_i \big\|_{\wedge} =
    \big(\det( \langle x_i,x_j \rangle_H )_{i,j=1}^n \big)^{1/2}. 
  \end{equation*}
\end{proposition}
\begin{proof}
  The semidefiniteness and the symmetry are obvious consequences of
  Proposition \ref{propcons}. It remains to prove the simplicial inequality.
  For given $x_i \in S_H$, $i=1,\ldots,n$ we write $y \in S_H$
  with suitable coefficients $c,c_i(i=1,\ldots,n)$ as
\begin{equation} \label{yrep}
  y = \sum_{j=1}^n c_j x_j+ c y^{\perp}, \quad \langle x_i, y^{\perp}\rangle_H=0
 \ (i=1,\ldots,n), \quad \|y^{\perp}\|_H = 1.
\end{equation}
Then we have the equality
  \begin{equation*}\label{ynorm1}
    1 = \langle y, y \rangle_H = \sum_{i,j=1}^n c_i c_j \langle x_i, x_j \rangle_H
    +c^2.
  \end{equation*}
 Using $|\langle x_i ,x_j \rangle_H| \le 1$ we obtain
  \begin{equation} \label{csum}
    1 \le \sum_{i,j=1}^n|c_i| |c_j|+ c^2 = \Big( \sum_{i=1}^n |c_i| \Big)^2
    + c^2.
  \end{equation}
  From \eqref{yrep} and the properties of the exterior product we have for $i=1,\ldots,n$
  \begin{align*}
    & x_1  \wedge \ldots \wedge x_{i-1} \wedge y \wedge x_{i+1} \wedge \ldots
    \wedge x_n \\ & =  x_1 \wedge \ldots \wedge x_{i-1} \wedge (c_i x_i + cy^{\perp}) \wedge x_{i+1} \ldots \wedge x_n.
  \end{align*}
  Further, the orthogonality in \eqref{yrep} implies via \eqref{eqinpro}
  \begin{align*}
    \big\langle x_1  \wedge \ldots \wedge x_n, x_1 \wedge \ldots \wedge x_{i-1} \wedge y^{\perp} \wedge x_{i+1} \ldots \wedge x_n \big\rangle=0,
  \end{align*}
  hence by \eqref{yrep}
  \begin{equation} \label{eqywedge}
    \|x_1\wedge \ldots \wedge x_{i-1} \wedge y \wedge x_{i+1} \wedge \ldots
    \wedge x_n\|_{\wedge}^2  = c_i^2 d_n^2 + c^2 d_{i,n-1}^2,
     \end{equation}
  where we use the abbreviations  $d_{i,n-1}=\|x_1\wedge \ldots \wedge x_{i-1} \wedge x_{i+1}\wedge \ldots \wedge x_n \|_{\wedge}$ and $d_n = \| x_1 \wedge \ldots \wedge x_n\|_{\wedge}$.
  .
  
  Note that Hadamard's inequality \eqref{eqhad} implies $d_n \le d_{i,n-1}$ for
  $i=1,\ldots,n$. With \eqref{csum}, \eqref{eqywedge} and the triangle inequality
  in $\R^2$ we find
  \begin{align*}
    d_n  & \le d_n \Big\|\begin{pmatrix} \sum_{i=1}^n |c_i| \\ |c| \end{pmatrix}\Big\|_2
    \le d_n \Big\|\begin{pmatrix} \sum_{i=1}^n |c_i| \\ n |c| \end{pmatrix}\Big\|_2
    \le d_n \sum_{i=1}^n \Big\|\begin{pmatrix} |c_i| \\ |c| \end{pmatrix}\Big\|_2
    \\
    & \le \sum_{i=1}^n \big(  c_i^2 d_n^2 + c^2 d_{i,n-1}^2 \big)^{1/2}
    = \sum_{i=1}^n d(x_1,\ldots,x_{i-1},y,x_{i+1},\ldots,x_n).
  \end{align*}
  
\flushright \end{proof}
\begin{example} \label{example3}
  As examples we list the
  induced pseudo $2$-metric  (compare \eqref{hilbert2}) and the pseudo $3$-metric on $S_H$:
  \begin{equation*} \label{eqexd2}
  \begin{aligned}
    d(u,v) & = |\sin \ang(u,v) |, \quad u,v \in S_H,
  \end{aligned}
  \end{equation*}
  \begin{equation} \label{eqexd3}
    \begin{aligned}
    d(u,v,w)^2& = \det \begin{pmatrix} 1 & \langle u,v\rangle_H &
      \langle u,w \rangle_H \\
      \langle u,v \rangle_H & 1 & \langle v,w \rangle_H \\
      \langle u,w \rangle_H & \langle v,w \rangle_H & 1 \end{pmatrix}
    \\
    & =  \det \begin{pmatrix} 1 & \cos \ang( u,v)
      &
      \cos \ang( u,w ) \\
      \cos \ang( u,v ) & 1 & \cos \ang( v,w ) \\
      \cos \ang( u,w ) & \cos \ang( v,w ) & 1 \end{pmatrix}
    \\
    &= 1 - \cos^2 \ang( u,v ) - \cos^2 \ang( u,w )-  \cos^2 \ang( v,w )
    \\ &  + 2  \cos \ang( u,v ) \cos \ang( u,w ) \cos \ang( v,w ),
    \quad u,v,w \in S_H.
  \end{aligned}
    \end{equation}
  
  \end{example}
  \begin{remark}
  The value $d(u,v,w)$ in \eqref{eqexd3} agrees with the three-dimensional polar sine
  $\prescript{3}{}{\mathrm{polsin}} (O,u\ v \ w)$, already defined by Euler,
  see \cite[Sect.6]{Eriksson78}. The work \cite{Eriksson78} discusses
  various ways of measuring $n$-dimensional angles and its history
  in spherical geometry. The main result is a simple derivation of
  the law of sines for the $n$-dimensional sine
  $\prescript{n}{}\sin$ and the $n$-dimensional polar sine
  $\prescript{n}{}{\mathrm{polsin}}$ defined by
  \begin{equation} \label{polsine}
  \begin{aligned}
    \prescript{n}{}\sin(O,x_1\cdots x_n)&= \frac{\|x_1 \wedge \ldots \wedge x_n\|^{n-1}}{\prod_{i=1}^n \|x_1 \wedge \ldots \wedge x_{i-1} \wedge x_{i+1}\wedge x_n||},\\
    \prescript{n}{}{\mathrm{polsin}}(O,x_1\cdots x_n)&= \frac{\|x_1 \wedge \ldots \wedge x_n\|}{\prod_{i=1}^n \|x_i\|}.
  \end{aligned}
  \end{equation}
  However, the simplicial inequality in Proposition \ref{metricball} seems not
  to have been observed.
\end{remark}

Next we consider  the Hilbert space $H=L(\R^k,\R^m)$
of linear mappings from $\R^k$ to $\R^m$ endowed with the Hilbert-Schmidt
(or Frobenius) inner product and norm
\begin{equation} \label{eqHS}
  \begin{aligned}
    \langle A, B \rangle_H & = \frac{1}{k}\sum_{j=1}^k \langle A e_j, Be_j\rangle_2, \quad
    A,B \in H, \\
    \|A\|_H^2 & = \frac{1}{k}\sum_{j=1}^k \|Ae_j\|_2^2, \quad A \in H.
  \end{aligned}
\end{equation}
Here the vectors $e_j,j=1,\ldots,k$ form an orthonormal basis of $\R^k$
and the prefactor $\frac{1}{k}$ is used for convenience, so that an
orthogonal map $A\in H$,  i.e. $A^{\star}A=I_k$, satisfies $\|A\|_H=1$. 
It is well known that the Hilbert-Schmidt inner product and norm  are
independent of the choice of orthonormal basis $(e_j)_{j=1}^k$.
By Proposition \ref{metricball} the unit sphere $S_H$ in $H$ carries a
pseudo $n$-metric defined by
 \begin{equation} \label{nmetricH}
    d_H(A_1,\ldots,A_n)= 
    \big(\det( \langle A_i,A_j \rangle_H )_{i,j=1}^n \big)^{1/2}. 
  \end{equation}
  Using Proposition \ref{prop1.1}(i) we arrive at the following corollary.
  \begin{corollary}\label{cor2.1}
    For $k \le m$ equation \eqref{nmetricH} defines a pseudo $n$-metric
    on the Stiefel manifold
    \begin{equation*} \label{eqstiefel}
      \mathrm{St}(k,m) = \{ A \in L(\R^k,\R^m): A^{\star}A = I_k \}.
    \end{equation*}
  \end{corollary}
  \begin{example} \label{example4}
    In case $n=2$ we obtain
    \begin{equation} \label{stiefex2}
      d_H(A_1,A_2) = (1 - \langle A_1,A_2 \rangle_H^2)^{1/2}, \quad A_1,A_2 \in
      \mathrm{St}(k,m).
    \end{equation}
  \end{example}

For the  Grassmannian $\cG(k,m)$ of $k$-dimensional subspaces of $\R^m$,
$k \le m$, the situation is not so simple.
One can identify $\cG(k,m)$ with the quotient space
\begin{align} \label{Gaqs}
  \mathrm{St}(k,m)/\sim, \quad \text{where} \; A \sim B \Longleftrightarrow
  \exists Q \in O(k): A=BQ
\end{align}
by setting $V=\mathrm{range}(A)$ for $[A]_{\sim}\in \mathrm{St}(k,m)/\sim$.  Then the Hilbert Schmidt norm is invariant w.r.t. the equivalence class, but
the inner product \eqref{eqHS} is not, since we obtain terms
$\langle A Q_1e_j,  BQ_2 e_j\rangle_2$ for which the orthogonal maps
$Q_1,Q_2$ differ, in general.

Therefore, we associate with every element $V \in \cG(k,m)$ the orthogonal
projection $P$ onto $V$, given by $P=A A^{\star}$ where $V=\mathrm{range}(A)$ for $A\in \mathrm{St}(k,m)$. The projection is an invariant for the
equivalence classes in \eqref{Gaqs}.
It is appropriate to measure orthogonal projections of rank $k$ by the scaled Hilbert-Schmidt inner product
  and norm
  \begin{equation*} \label{eqkHS}
  \begin{aligned}
    \langle A, B \rangle_{k,H} & = \frac{1}{k}\sum_{j=1}^m
    \langle A e_j, Be_j\rangle_2, \quad
    A,B \in L(\R^m,\R^m), \\
    \|A\|_{k,H}^2 & = \frac{1}{k}\sum_{j=1}^m \|Ae_j\|_2^2, \quad A \in L(\R^m,\R^m),
  \end{aligned}
\end{equation*}
  where $(e_j)_{j=1,\ldots,m}$ form an orthonormal basis of $\R^m$.
  We claim that the orthogonal projection $P$ belongs to the unit ball
  \begin{align*}
    S_{k,H}= \{ B \in L(\R^m,\R^m): \|B\|_{k,H}=1 \}.
  \end{align*}
  To see this, choose a  special orthonormal basis of $\R^m$ where
  $e_1,\ldots,e_k$ are the columns of $A$ and $e_{k+1},\ldots,e_m$ form
  a basis of the orthogonal complement $V^{\perp}$. Since the norm is
  invariant and $AA^* e_j=e_j$ for $j=1,\ldots,k$, $AA^{\star}e_j=0$ for
  $j >k$ we obtain $\|P\|_{k,H}=1$. Thus we can proceed as before
  and invoke Propositions \ref{metricball} and \ref{prop1.1}(i) to obtain
  the following result.

\begin{proposition} \label{cor2.2}
  For $V_i\in \cG(k,m)$ let $P_i\in L(\R^m,\R^m)$, $i=1,\ldots,n$ be the corresponding orthogonal projections. Then the setting
   \begin{equation} \label{nmetricproj}
    d_{\cG}(V_1,\ldots,V_n)= 
    \big(\det( \langle P_i,P_j \rangle_{k,H} )_{i,j=1}^n \big)^{1/2}.
  \end{equation}
   defines a pseudo $n$-metric on $\cG(k,m)$.
\end{proposition}

For an interpretation of \eqref{nmetricproj} let us compute the inner product
of two orthogonal projections.
\begin{lemma} \label{inprodorth}
  For $A_j \in \mathrm{St}(k,m)$, $j=1,2$ let $P_j=A_j A_j^{\star}$ be
  the corresponding orthogonal projections in $\R^m$. Then the following holds
  \begin{equation*} \label{expressorth}
    \langle P_1,P_2 \rangle_{k,H} = \frac{1}{k}\sum_{j=1}^k \sigma_j^2,
  \end{equation*}
  where $\sigma_1\ge \ldots \ge \sigma_k \ge 0$ are the singular values
  of $A_1^{\star}A_2$.
\end{lemma}
\begin{proof}
  Consider the SVD  $A_1^{\star} A_2 = Y \Sigma Z^{\star}$ with $\Sigma
  = \diag(\sigma_1,\ldots,\sigma_k)$ and $Y,Z\in O(k)$. Further, choose
  $A_j^c \in \mathrm{St}(m-k,m)$ such that $ \begin{pmatrix} A_j & A_j^c
  \end{pmatrix}$, $j=1,2$ are orthogonal. Then we conclude
  \begin{align*}
   k \langle P_1, P_2 \rangle_{k,H} & = \mathrm{tr}(A_1 A_1^{\star}A_2 A_2^{\star})
    = \mathrm{tr}\Big( \begin{pmatrix} A_1 & A_1^c \end{pmatrix}
    \begin{pmatrix} A_1^{\star} A_2 & 0 \\ 0 & 0 \end{pmatrix}
    \begin{pmatrix} A_2^{\star} \\ A_2^{c \star} \end{pmatrix} \Big)\\
  & =  \mathrm{tr}\Big( 
    \begin{pmatrix} A_1^{\star} A_2 & 0 \\ 0 & 0 \end{pmatrix}
    \begin{pmatrix} A_2^{\star} \\ A_2^{c \star} \end{pmatrix} \begin{pmatrix} A_1 & A_1^c \end{pmatrix} \Big)\\
    &  =\mathrm{tr} \begin{pmatrix} A_1^{\star}A_2 A_2^{\star}A_1 & A_1^{\star}A_2
      A_2^{\star} A_1^c \\ 0 & 0 \end{pmatrix}=
    \mathrm{tr} \big(A_1^{\star}A_2 A_2^{\star}A_1 \big) \\
    & = \mathrm{tr}(Y \Sigma Z^{\star} Z \Sigma Y^{\star}) = \mathrm{tr} \Sigma^2
    = \sum_{j=1}^k \sigma_j^2.
  \end{align*}
  \flushright \end{proof}
Recall from \cite[Ch.6.4]{GvL2013} that  the singular values $\sigma_1 \ge \ldots \ge \sigma_k$ of $A_1^{\star}A_2$ define the principal angles
$0 \le \theta_1 \le \ldots \le \theta_k \le \frac{\pi}{2}$ between $V_1$ and $V_2$ by $\sigma_j=\cos(\theta_j)$, $j=1,\ldots,k$.
   
For $n=2$ this leads to  an explicit expression of \eqref{nmetricproj}.
\begin{proposition}\label{lemGn=2}
  In case $n=2$ equation \eqref{nmetricproj} defines a $2$-metric on $\cG(k,m)$
  by
  \begin{equation} \label{explicitn=2}
    d_{\cG}(V_1,V_2)=  (1-\frac{1}{k^2} \big(\sum_{j=1}^k \sigma_j^2)^2 \big)^{1/2}
    = \frac{1}{k} \Big( k^2 - \big( \sum_{j=1}^k \cos^2(\theta_j) \big)^2 \Big)^{1/2},
  \end{equation}
  where $0\le \theta_1\le \cdots \le \theta_k \le \frac{\pi}{2}$
  are the principal angles of the subspaces $V_1$ and $V_2$.
\end{proposition}
\begin{proof}
  From Lemma \ref{inprodorth} and \eqref{nmetricproj} we find
  \begin{align*}
    d_{\cG}(V_1,V_2)& = (1-\frac{1}{k^2} \big(\sum_{j=1}^k \sigma_j^2)^2 \big)^{1/2}
    = \frac{1}{k} \Big( k^2 - \big( \sum_{j=1}^k \cos^2(\theta_j) \big)^2 \Big)^{1/2}.
  \end{align*}

  \flushright \end{proof}

In Section \ref{sec6.1} we continue the discussion  of 
pseudo $n$-metrics on the Grassmannian and its relation to known $2$-metrics.

\section{A pseudo $n$-metric on hypergraphs}
\label{sec5}
The notion of hypergraph allows edges which connect more than two vertices;
see \cite{Br13}.
\begin{definition} \label{defhyper}
  A pair $(V,E)$ is called a hypergraph if $V$ is a finite set 
  and $E$ is a subset of the power set $\mathcal{P}(V)$ of $V$.
  An element $e \in E$ is called a hyperedge. In particular,
  it is called an $n$-hyperedge if $\#e=n$.
    The hypergraph is called $n$-uniform if all its hyperedges are $n$-hyperedges.
    \end{definition}
  Obviously, a $2$-uniform hypergraph is an ordinary (undirected)
 graph. The following definition generalizes the standard notion of connectedness.
\begin{definition} \label{defconn}
  Let $(V,E)$ be an $n$-uniform  hypergraph.
  \begin{itemize}
    \item[(i)]
      A subset $P \subset E$ is called a connected component of $(V,E)$
      if  for any two $e_0,e_{\infty} \in P$ there exist 
      finitely many $e_1,\ldots,e_k \in P$ such that  
      $e_{i-1}\cap e_{i} \neq \emptyset$ for $i=1,\ldots, k+1$ where
      $e_{k+1}:=e_{\infty}$.
      Any subset $W \subset \bigcup_{e \in P}e$ is said to be connected by $P$.
\item[(ii)] The hypergraph is called connected if
  $E$ is a connected component of $(V,E)$ and $V=\bigcup_{e \in E}e$. 
  \end{itemize}
\end{definition}
  In case $n=2$  this agrees with the usual notion of connected components.
  For a connected hypergraph one can connect any subset of $V$ by
  the (maximal) connected component $E$.

\begin{definition}\label{defpseudo}
  Let $(V,E)$ be an $n$-uniform  and connected hypergraph. For
  any tuple $(v_1,\ldots,v_n)\in V^n$ define $d(v_1,\ldots,v_n)=0$
  if there exist $i,j \in \{1,\ldots,n\}$ with $i \neq j$ and  $v_i=v_j$.
  Otherwise set $W=\{v_1,\ldots,v_n\}$ and
  \begin{equation*} \label{eqpseudo}
    d(v_1,\ldots,v_n)
      =    \min \{\#P: P \text{ connected component of $(V,E)$ connecting} \; W \}.
  \end{equation*}
\end{definition}

\begin{example} Let $V=\{1,2,3,4\}$ and
  \begin{align*} E=\big\{\{1,2,4\},\{2,3,4\},\{1,3,4\} \big\}.
  \end{align*}
  This is a $3$-uniform hypergraph satisfying
  $ d(1,2,4)=1$, $d(2,3,4)=1$, $d(1,3,4)=1$. Moreover, we have
  $ d(1,2,3)=2$ since
  $\{1,2,3\}$ is
  not a hyperedge and $P=\big\{\{1,2,4\},\{2,3,4\}\big\}$ is a connected
  component of $(V,E)$ which has cardinality $2$ and connects $\{1,2,3\}$.
  Another connected component with this property and of the same cardinality  is
  $P'=\big\{\{1,2,4\},\{1,3,4\}\big\}$.
    \end{example}

\begin{proposition} \label{prophyper}
  Let $(V,E)$ be an $n$-uniform and connected hypergraph. Then
  the map $d:V^n \to \R$ from Definition \ref{defpseudo} is a
  pseudo $n$-metric on $V$.
\end{proposition}
\begin{remark} In case $n=2$ this generates the topology on hypergraphs
  discussed in \cite{DR21}.
  \end{remark}
\begin{proof}
  The semidefiniteness and symmetry follow directly from  Definition
  \ref{defpseudo}.
  Now consider $(v_1,\ldots,v_n) \in V^n$. If two components agree then
  \eqref{tri} is trivially satisfied. Otherwise $W=\{v_1,\ldots,v_n\} $
  is a subset of $V$ with $\#W=n$. Let $y \in V$ be arbitrary.
  By Lemma \ref{lempos} it suffices to prove the simplicial inequality
  for  $y \notin W$.
   Let $P_1$ resp. $P_2$ be connected components of minimal cardinality
   $p_1=d(y,v_2,\ldots,v_n)$ resp. $p_2=d(v_1,y,\ldots,v_n)$
   which connect $W_1=\{y,v_2,\ldots,v_n\}$ resp. $W_2=\{v_1,y,\ldots,v_n\}$.
   Then $P=P_1 \cup P_2$ is a connected component which connects $\{v_1,\ldots,v_n\}$. To see this, note  first that
   \begin{align*}
     W \subset W_1 \cup W_2 \subset \bigcup_{e \in P_1}e \cup \bigcup_{e \in P_2}e
     =\bigcup_{e \in P} e.
   \end{align*}
   Let $e_0,e_{\infty}\in P$. If both are in $P_1$ or in $P_2$ then they can
   be connected as in Definition \ref{defconn} (i). So assume w.l.o.g.
   $e_0 \in P_1$, $e_{\infty} \in P_2$. Since $P_1$ connects $W_1$ there
   exists a hyperedge $e_+ \in P_1$ such that $y \in e_+$ and a hyperedge $e_- \in
   P_2$ such that $y \in e_-$. Further, there are sequences of
   hyperedges $e_1,\ldots,e_k\in P_1$ such that $e_{i-1}\cap e_i \neq \emptyset$
   for $i=1,\ldots,k+1$ with $e_{k+1}=e_+$. Similarly, there exist
   hyperedges $e_-=e_0',e_1',\ldots,e_{k'}'$ such that
   $e'_{i-1}\cap e'_i \neq \emptyset$ for $i=1,\ldots,k'+1$ with $e'_{k'+1}=
   e_{\infty}$. Since $y \in e_{k+1}\cap e'_0$ we obtain that the sequence
   $e_0,\ldots,e_k,e_+,e_-,e'_1,\ldots,e'_{k'},e_{\infty}$ lies in $P$ and
   has nonempty intersections of successive hyperedges. Therefore,
   \begin{align*}
     d(v_1,\ldots,v_n) \le p_1 + p_2 =d(y,v_2,\ldots,v_n)+d(v_1,y,\ldots,v_n),
   \end{align*}
   which proves \eqref{tri}. \flushright \end{proof}

The proof shows that we have a sharper inequality than \eqref{tri}
 \begin{align*} \label{tri}
   d(x_1,\ldots,x_n) \le \max_{1\le i < j \le n}&\big( d(x_1,\ldots,x_{i-1},y,x_{i+1},\ldots,x_n)\\
   &+ d(x_1,\ldots,x_{j-1},y,x_{j+1},\ldots,x_n)\big).
   \end{align*}

 \section{Further Examples, Problems, and Conjectures}
 \label{sec6}
 \subsection{Generalized Vandermonde pseudo $n$-metrics}
 \label{sec6.0}
 
 It is tempting to conjecture that 
 \begin{equation} \label{gendef}
  d(x_1,\ldots,x_n) = \prod_{1 \le i < j \le n} \|x_i-x_j\|,
  \quad x_i\in X, i=1,\ldots, n,
  \end{equation}
 generates an $n$-metric on a Euclidean space
 $(X, \langle \cdot,\cdot\rangle, \|\cdot\|)$.
    
    However, the expression \eqref{gendef} does not satisfy the simplicial inequality
    in dimensions $\ge 3$ as the following example shows.
 \begin{example}
   Consider $X=\R^3$, $y=0$, and let $x_1,\ldots,x_4$ be the vertices
    of an equilateral tetrahedron:
    \begin{align*}
      x_1&=(1,0,0), \quad x_2 = \frac{1}{3}(-1,2\sqrt{2},0), \\
      x_3& =\frac{1}{3}(-1,-\sqrt{2},\sqrt{6}), \quad
      x_4= \frac{1}{3}(-1,-\sqrt{2}, -\sqrt{6}).
    \end{align*}
    Then one verifies
    \begin{align*}
      \|x_j\|&=1 , \quad j=1,\ldots,4, \\
      \|x_i -x_j\|& =  \sqrt{\frac{8}{3}}, \quad 1\le i < j \le 4.
    \end{align*}
    The simplicial inequality then requires
    \begin{align*}
      \big( \frac{8}{3}\big)^3 \le 4 \big( \frac{8}{3}\big)^{3/2}
      \quad \text{or} \quad 2^5 \le 3^3,
    \end{align*}
    which is wrong.
 \end{example}

 In the following we pursue another highdimensional generalization of  the Vandermonde expression  \eqref{Cprodn}
which uses multilinear symmetric maps.
It is based on the following purely algebraic result.
\begin{theorem} (The general Vandermonde equalities) \label{th:gV}\\
    Let $X,Y$ be vector spaces over $\bbK=\R,\C$ and let
    \begin{align*}
      A\colon X^{M_n}\longrightarrow Y, \quad M_n = \frac{1}{2}n(n-1), n \ge 2,
    \end{align*}
    be an $M_n$-linear and symmetric map. 
    \begin{itemize}
    \item[(i)] For all $x_1,\ldots,x_n \in X$ the following holds:
      \begin{equation} \label{eq:expandV}
        A\Big(\prod_{1\le j < i \le n}(x_i-x_j)\Big)=
        \sum_{\pi\in \mathcal{P}_n} \mathrm{sign}(\pi) A\Big( \prod_{j=1}^n
        x_j^{\pi(j)-1}\Big).
      \end{equation}
    \item[(ii)] The quantity
      \begin{align*}
        V(x_1,\ldots,x_n) = A\Big(\prod_{1\le j < i \le n}(x_i-x_j)\Big)
      \end{align*}
     satisfies for all $\xi \in X$ and $x_1,\ldots,x_n \in X$ the following equality:
      \begin{equation} \label{eq:sumV}
        V(x_1,\ldots,x_n) = \sum_{i=1}^n V(x_1,\ldots,x_{i-1},\xi,x_{i+1},\ldots,x_n).
      \end{equation}
    \end{itemize}
\end{theorem}
\begin{proof} First note that it is convenient to write the arguments of the symmetric map $A$ as products
  since the sequence of arguments does not matter.\\
    {\bf (i):} \\
  For $n=2$ the equation \eqref{eq:expandV} is obvious:
  \begin{align*}
    \mathrm{sign}(1,2) A(x_2^{\pi(2)-1})+ \mathrm{sign}(2,1)A(x_1^{\pi(1)-1})
    = A x_2 - A x_1 = A(x_2-x_1).
  \end{align*}
  For the induction step from $n$ to $n+1$ we use the symmetry of $A$:
\begin{equation} \label{eq:bigsum}
  \begin{aligned}
    & A\Big(\prod_{1\le j < i \le n+1}(x_i-x_j)\Big) = A\Big(\prod_{1\le j < i \le n}(x_i-x_j),\prod_{j=1}^n(x_{n+1}-x_j)\Big)\\
    & = \sum_{\pi \in \mathcal{P}_n}\mathrm{sign}(\pi)A\Big(\prod_{j=1}^n x_j^{\pi(j)-1},\prod_{j=1}^n(x_{n+1}-x_j)\Big)\\
    & = \sum_{\pi \in \mathcal{P}_n}\mathrm{sign}(\pi) \sum_{\sigma \subseteq
      \{1,\ldots,n\}}(-1)^{|\sigma|}A\Big(\prod_{j=1}^n x_j^{\pi(j)-1},
    \prod_{j \in \sigma}x_j ,x_{n+1}^{n - |\sigma|} \Big).
  \end{aligned}
  \end{equation}
  We consider the summand for any two indices $\ell,k\in \{1,\ldots,n\}$ with
  $\pi(\ell)= \pi(k)+1$ (hence $\ell \neq k$)  and any
  $\sigma \subseteq\{1,\ldots,n\}$ with $k \in \sigma, \ell \notin \sigma$.
  Let $\tau_{k,\ell}$ be the transposition of $k$ and $\ell$ and set
  \begin{align*}
    \tilde{\pi} & = \pi \circ \tau_{k,l}, \quad \tilde{\sigma}=(\sigma \setminus
    \{k\})\cup \{\ell\}.
  \end{align*}
  Then we obtain $|\tilde{\sigma}|=|\sigma|$, $\mathrm{sign}(\tilde{\pi})=
  - \mathrm{sign}(\pi)$, $\tilde{\pi}(k)=\pi(\ell)$, $\tilde{\pi}(\ell)=\pi(k)$ and
 
  \begin{align*}
    & \mathrm{sign}(\tilde{\pi})(-1)^{|\tilde{\sigma}|}A\Big(\prod_{j=1}^n x_j^{\tilde{\pi}(j)-1},
    \prod_{j \in \tilde{\sigma}}x_j ,x_{n+1}^{n - |\tilde{\sigma}|} \Big)\\
    & = -\mathrm{sign}(\pi)(-1)^{|\sigma|} A\Big( \prod_{j=1,j\neq k,\ell}^n
    x_j^{\pi(j)-1}, x_k^{\pi(\ell)-1}, x_{\ell}^{\pi(k)-1}, x_{\ell},
    \prod_{j\in \sigma,j\neq k}x_j, x_{n+1}^{n - |\sigma|}\Big)\\
    &  = -\mathrm{sign}(\pi)(-1)^{|\sigma|} A\Big( \prod_{j=1,j\neq k,\ell}^n
    x_j^{\pi(j)-1}, x_k^{\pi(k)}, x_{\ell}^{\pi(\ell)-1},
    \prod_{j\in \sigma,j\neq k}x_j, x_{n+1}^{n - |\sigma|}\Big)\\
    & = -\mathrm{sign}(\pi)(-1)^{|\sigma|} A\Big( \prod_{j=1}^nx_j^{\pi(j)-1},
    \prod_{j\in \sigma}x_j, x_{n+1}^{n - |\sigma|}\Big).
      \end{align*}
  The last term belongs to $\pi$ and $\sigma$, so that these two terms
  cancel each other in \eqref{eq:bigsum}.
  It is left to consider $(\pi,\sigma)$-terms for which there
  is no index $j \in \{1,\ldots,n\}$ in the sequence
  $\pi^{-1}(n),\pi^{-1}(n-1),\ldots,\pi^{-1}(1)$ such that $\pi^{-1}(j) \notin \sigma$ and $\pi^{-1}(j-1) \in \sigma$. For a given $\pi\in \mathcal{P}(n)$, the only sets
  $\sigma$ satisfying this condition are
  \begin{align*}
    \sigma_p=\{\pi^{-1}(n),\ldots,\pi^{-1}(p)\} \quad \text{for some} \;\;
    1\le p \le n+1,
  \end{align*}
  where we set $\sigma_{n+1}=\emptyset$. Therefore, we obtain from \eqref{eq:bigsum}
  \begin{equation} \label{eq:smallsum}
    \begin{aligned}
      & A\Big(\prod_{1\le j < i \le n+1}(x_i-x_j)\Big)\\
      &= \sum_{\pi \in \mathcal{P}_n}
    \sum_{p=1}^{n+1} \mathrm{sign}(\pi) (-1)^{n+1-p}
     A \Big( \prod_{j=1}^nx_j^{\pi(j)-1},\prod_{j \in \sigma_p} x_{j}, x_{n+1}^{p-1}\Big).
    \end{aligned}
  \end{equation}
  For $p \in \{1,\ldots,n+1\}$ and $\pi \in \mathcal{P}_n$ define
  $\tilde{\pi} \in \mathcal{P}_{n+1}$ by
  \begin{align*}
    \tilde{\pi}(j) = \begin{cases} \pi(j), & \pi(j) \in \{1,\ldots,p-1\},
      j \le n, \\
      \pi(j)+1, & \pi(j) \in \{p,\ldots,n\}, j\le n, \\
      p, & j=n+1.
    \end{cases}
  \end{align*}
  Clearly, the map $\mathcal{P}_n\times \{1,\ldots,n+1\} \to \mathcal{P}_{n+1}$,
  $(\pi,p) \mapsto \tilde{\pi}$ is a bijection which satisfies
  $\mathrm{sign}(\tilde{\pi})= (-1)^{n+p-1}\mathrm{sign}(\pi)$. With this setting
  we may write \eqref{eq:smallsum} as
  \begin{align*}
    A\Big(\prod_{1\le j < i \le n+1}(x_i-x_j)\Big)=
    \sum_{\tilde{\pi}\in \mathcal{P}_{n+1}} \mathrm{sign}(\tilde{\pi})
    A\Big( \prod_{j=1}^{n+1} x_j^{\tilde{\pi}(j)-1}\Big),
  \end{align*}
  which is our assertion.

  {\bf (ii):} \\
  Consider the right-hand side of \eqref{eq:sumV},use formula \eqref{eq:expandV}, and sort by the powers of $\xi$ to obtain
  \begin{align*}
     \sum_{i=1}^n& V(x_1,\ldots,x_{i-1},\xi,x_{i+1},\ldots,x_n) 
    = \sum_{i=1}^n \sum_{\pi \in \mathcal{P}_n} \mathrm{sign}(\pi)
    A \Big( \prod_{j=1, j\neq i}^n x_j^{\pi(j)-1}, \xi^{\pi(i)-1} \Big)\\
    & \quad \quad \qquad = \sum_{p=1}^n \sum_{\pi \in \mathcal{P}_n} \mathrm{sign}(\pi)
    A \Big( \prod_{j=1, \pi(j)\neq p}^n x_j^{\pi(j)-1}, \xi^{p-1} \Big).
  \end{align*}
  We consider two cases:\\
  {\bf $p=1$: } \\
    In this case we have
    \begin{align*}
      \sum_{\pi \in \mathcal{P}_n} \mathrm{sign}(\pi)
      A \Big( \prod_{j=1, \pi(j)\neq 1}^n x_j^{\pi(j)-1}, \xi^{p-1} \Big)&=
       \sum_{\pi \in \mathcal{P}_n} \mathrm{sign}(\pi)
      A \Big( \prod_{j=1}^n x_j^{\pi(j)-1} \Big)\\
      &= V(x_1,\ldots,x_n).
    \end{align*}

    \noindent
    {\bf $2 \le p \le n$:}\\
    We show that the corresponding terms vanish.
    Let $\ell=\pi^{-1}(p)$, $k =\pi^{-1}(1)$ (note $k \neq \ell$) and 
     define $\tilde{\pi}= \pi \circ \tau_{kl}$ as above, so
    that
    \begin{align*}
      \mathrm{sign}(\tilde{\pi})= - \mathrm{sign}(\pi), \quad
        \tilde{\pi}(\ell)=1, \quad \tilde{\pi}(k) = p.
    \end{align*}
    The corresponding $\tilde{\pi}$-term then satisfies
    \begin{align*}
      \mathrm{sign}(\tilde{\pi}) A\Big( \prod_{j=1, \tilde{\pi}(j) \neq p}^n
      x_j^{\tilde{\pi}(j)-1},  \xi^{p-1} \Big)= -  \mathrm{sign}&(\pi) A\Big(\prod_{j=1,j \neq k,\ell}^n
      x_j^{\pi(j)-1}, x_{\ell}^{\tilde{\pi}(\ell)-1}, \xi^{p-1} \Big)\\
    = - \mathrm{sign}(\pi)A\Big(\prod_{j=1,j \neq \ell}^n
      x_j^{\pi(j)-1},x_k^{\pi(k)-1}, \xi^{p-1} \Big) & = - \mathrm{sign}(\pi)A\Big(\prod_{j=1,j \neq \ell}^n
      x_j^{\pi(j)-1}, \xi^{p-1} \Big).
    \end{align*}
    Hence the terms which belong to $\pi$ and $\tilde{\pi}$ cancel each other, and
    the summand corresponding to $p \in \{ 2,\ldots,n\}$ vanishes.
      \end{proof}
An easy consequence of Theorem \ref{th:gV} is the following construction of
a pseudo $n$-metric.
\begin{corollary} (The generalized Vandermonde pseudo $n$-metric) \label{cor:Vnmet}\\
  Let $X$ be a vector space and let $(Y,\|\cdot\|)$ be a normed space. Any $M_n$-linear
  and symmetric map $A:X^{M_n}\to Y$ defines a pseudo $n$-metric on $X$ via
  \begin{equation} \label{eq:defVmet}
    d(x_1,\ldots,x_n)= \big\| A\big( \prod_{1\le j <i\le n}(x_i-x_j) \big) \|,
    \quad  x_1,\ldots,x_n \in X.
  \end{equation}
  \end{corollary}
\begin{proof} The semidefiniteness and the nonnegativity are  obvious from the definition \eqref{eq:defVmet}. The prove symmetry we apply the representation \eqref{eq:expandV} for every $\sigma \in \mathcal{P}_n$:
  \begin{align*}
    d(x_{\sigma(1)},\ldots,x_{\sigma(n)})& = \big\| \sum_{\pi\in \mathcal{P}_n} \mathrm{sign}(\pi) A\Big( \prod_{j=1}^n
      x_{\sigma(j)}^{\pi(j)-1}\Big) \big\| \\
      & = \big\| \sum_{\pi\in \mathcal{P}_n} \mathrm{sign}(\pi) A\Big( \prod_{j=1}^n
      x_{j}^{(\pi \circ \sigma^{-1})(j)-1}\Big) \big\|\\
      & = \big\|\mathrm{sign}(\sigma) \sum_{\pi\in \mathcal{P}_n} \mathrm{sign}(\pi\circ \sigma^{-1}) A\Big( \prod_{j=1}^n
      x_{j}^{(\pi \circ \sigma^{-1})(j)-1}\Big) \big\|\\
      & = \big\| \sum_{\tilde{\pi}\in \mathcal{P}_n} \mathrm{sign}(\tilde{\pi}) A\Big( \prod_{j=1}^n
      x_{j}^{\tilde{\pi}(j)-1}\Big) \big\|\\
      & = d(x_1,\ldots,x_n).
  \end{align*}
  Finally, the simplicial inequality follows by taking norms in \eqref{eq:sumV} and
  using the triangle inequality.
\end{proof}
The $n$-mterics from Theorem \ref{Cprop} are special cases of Corollary \ref{cor:Vnmet}
  for $\dim(X)=\dim(Y)=1,2$ when taking the multilinear form $A$ induced by real resp. complex multiplication.
  It will be of interest to see whether true $n$-metrics can be constructed via Corollary \ref{cor:Vnmet} for spaces
  with $\dim(X)\ge 3$.
  

 \subsection{More pseudo $n$-metrics on the Grassmannian $\cG(k,m)$}
 \label{sec6.1}
 The pseudo $n$-metric \eqref{nmetricproj} on $\cG(k,m)$ is not completely
 satisfactory for several reasons.
First, one expects a suitable pseudo $n$-metric on $\cG(k,m)$ 
to be a canonical generalization of the standard $2$-metric given by
\begin{equation*} \label{eqgrassmann}
  d(V_1,V_2)=  \|P_1 - P_2\|_2, \quad V_1,V_2 \in \cG(k,m),
\end{equation*}
where $P_1,P_2$ denote the orthogonal projections onto $V_1,V_2$ respectively,
and $\|\cdot \|_2$ denotes the Euclidean operator norm (spectral norm)
in $\R^m$; see \cite[Ch.6.4.3]{GvL2013}.  
This  $2$-metric may be expressed equivalently as
\begin{equation} \label{eqangle}
  d(V_1,V_2)= \sin (\ang(V_1,V_2)) = \sin(\theta_k),
\end{equation}
where $\theta_k$ is the largest principal angle between $V_1$ and $V_2$;
see \cite[Ch.2.5, Ch.6.4]{GvL2013} and Proposition \ref{lemGn=2}.
Note that the $2$-metric \eqref{eqangle} differs from the $2$-metric  in
\eqref{explicitn=2}.
An alternative to \eqref{nmetricproj} is to measure the exterior product
$\bigwedge_{j=1}^n P_j$ of orthogonal projections by a norm different from
the Hilbert-Schmidt norm (cf. \eqref{polsine}):
\begin{align*}
  \max_{0 \neq F_{\ell}\in \R^{m,m},\ell=1,\ldots,n}\frac{ |\Big(\bigwedge_{j=1}^n P_j\Big)(F_1,\ldots,F_n)|}{\prod_{j=1}^n \| F_j\|_*},
\end{align*}
where $\|\cdot\|_*$ is either the spectral norm $\|\cdot\|_2$ or its
dual $\|\cdot \|^D$ given by the sum of singular values; see
\cite[Ch.5.6]{HJ2013}.  However, neither were we able to prove
the simplicial inequality for this setting, nor did we find an
explicit expression for the associated $2$-metric comparable to
\eqref{eqangle}.

The fact that the Grassmannian may be viewed as a quotient space
of $\mathrm{St}(k,m)$ (see \eqref{Gaqs}) suggests another construction
of  a $2$-metric on $\cG(k,m)$  by setting
\begin{equation} \label{dHGrass}
  d_{\cG,H}(V_1,V_2) = \min_{Q_1,Q_2\in O(k)}d_H(A_1Q_1,A_2Q_2)=
  \min_{Q \in O(k)}d_H(A_1,A_2Q), \quad 
\end{equation}
for $V_j=\mathrm{range}(A_j)$, $A_j \in \mathrm{St}(k,m)$ and $d_H$ from \eqref{nmetricH},
\eqref{stiefex2}.
The symmetry and definiteness of
$d_{\cG,H}$ are obvious. The triangle inequality is also easily seen:
for $A_j \in \mathrm{St}(k,m)$ and $V_j=\mathrm{range}(A_j)$, $j=1,2,3$
 select $Q_1,Q_2 \in O(k)$  with
\begin{equation*} \label{chooseQ}
  d_H(A_1,A_3Q_1)=d_{\cG,H}(V_1,V_3), \quad d_{H}(A_3,A_2Q_2)= d_{\cG,H}(V_3,V_2).
\end{equation*}
Then we obtain
\begin{align*}
  d_{\cG,H}(V_1,V_2) & \le d_H(A_1,A_2 Q_2Q_1) \le d_{H}(A_1,A_3Q_1)+d_{H}(A_3Q_1,A_2Q_2Q_1) \\
  & =d_{H}(A_1,A_3Q_1)+d_{H}(A_3,A_2Q_2)=d_{\cG,H}(V_1,V_3)+ d_{\cG,H}(V_3,V_2).
\end{align*}
The next proposition gives an explicit expression of the $2$-metric \eqref{dHGrass}
in terms of the principal angles $\theta_j$, $j=1,\ldots,k$ between the subspaces $V_1$ and $V_2$.
\begin{proposition} \label{dGangle}
  The expression \eqref{dHGrass} defines a $2$-metric on the Grassmannian $\cG(k,m)$.
  The metric satisfies 
  \begin{equation}  \label{dGrep}
    d_{\cG,H}(V_1,V_2) = \Big(1 - k^{-2}\big(\sum_{j=1}^k \sigma_j \big)^2\Big)^{1/2},
  \end{equation}
  where $\sigma_1=\cos(\theta_1)\ge \cdots \ge \sigma_k=\cos(\theta_k)>0$
  denote the singular values of $A_1^{\star}A_2$.
\end{proposition}
\begin{proof}
  Consider the singular value decomposition $A_1^{\star}A_2= Y \Sigma Z^{\star}$ where $Y,Z \in O(k)$
  and $\Sigma = \mathrm{diag}(\sigma_1,\ldots,\sigma_k)$.
  By \eqref{stiefex2} the minimum in \eqref{dHGrass} is achieved when
  maximizing
  \begin{align*}
    \langle A_1Q_1, A_2Q_2\rangle_H^2& = k^{-2} \mathrm{tr}^2(Q_1^{\star}A_1^{\star}A_2 Q_2)
    = k^{-2}\mathrm{tr}^2(Q_1^{\star}Y \Sigma Z^{\star} Q_2) \\
   & = k^{-2}\mathrm{tr}^2(\Sigma Z^{\star} Q_2Q^{\star}_1 Y) \quad
    \text{w.r.t.} \quad Q_1,Q_2 \in O(k).
  \end{align*}
  Setting $Q=Z^{\star} Q_2Q^{\star}_1 Y \in O(k)$ this is equivalent to maximizing
  $\mathrm{tr}^2(\Sigma Q)$ with respect to $Q \in O(k)$.
  Since the diagonal elements of $Q\in O(k)$ satisfy $|Q_{jj}| \le 1$ we
  obtain
  \begin{align*}
    \mathrm{tr}^2(\Sigma Q)& = \big|\sum_{j=1}^k \sigma_j Q_{jj} \big|^2
    \le \big( \sum_{j=1}^k \sigma_j |Q_{jj}|)^2 \le \big( \sum_{j=1}^k \sigma_j \big)^2= \mathrm{tr}^2(\Sigma I_k).
    \end{align*}
  This proves our assertion.
    \flushright \end{proof}
Note that \eqref{dGrep} differs from both expressions \eqref{eqangle} and
\eqref{explicitn=2}. Let us further note that we did not succeed to prove the simplicial inequality for the natural generalization of \eqref{dHGrass}
\begin{equation*} \label{dHnGrass}
 d_{\cG,H}(V_1,\ldots,V_n) = \min_{Q_j\in O(k),j=1,\ldots,n}d_H(A_1Q_1,\ldots, A_nQ_n) 
\end{equation*}
with $V_j=\mathrm{range}(A_j)$, $A_j \in \mathrm{St}(k,m)$. We rather conjecture
that this is false which is suggested by analogy to a counterexample
from the next subsection.

Finally, we discuss the special case $k=1$, 
when we can identify $V\in \cG(1,m)$
with a unit vector $v \in S_{\R^m}$ such that $V = \mathrm{span}\{v\}$. 
Proposition \ref{metricball} then shows that a suitable pseudo $n$-metric
on $\cG(1,m)$ is given by
\begin{align*}
  d(V_1,\ldots,V_n)= \big(\det(\langle v_i, v_j \rangle_2)_{i,j=1}^n \big)^{1/2},
\end{align*}
where $V_i=\mathrm{span}\{v_i\}$, $\|v_i\|_2=1$ for $i=1,\ldots,n$.
In particular, in case $n=2$ we obtain
\begin{equation*}
  d(V_1,V_2)=\big( 1 -\langle v_1,v_2\rangle^2\big)^{1/2} =
  \big( 1 -\cos\ang(v_1,v_2)^2\big)^{1/2}= \sin \ang(v_1,v_2),
\end{equation*}
which is consistent with \eqref{eqangle}.
 On the other hand the pseudo $n$-metric \eqref{nmetricproj} leads to
 \begin{equation*}
   d_{\cG}(V_1,\ldots,V_n)= \big(\det(\langle v_i,v_j \rangle_2^2 )_{i,j=1}^n \big)^{1/2}
 \end{equation*}
 which in case $n=2$ differs from \eqref{eqangle}.
 Therefore, we still consider it an open problem to construct a pseudo $n$-metric
 on the Grassmannian which is a natural generalization
 of the $2$-metric \eqref{eqangle}.
 
  \subsection{Pseudo $n$-metric for subsets}
  \label{sec6.2}
  It is natural to ask whether the classical Hausdorff distance
  for closed sets of a metric space (see e.g.\ \cite[Chapter 1.5]{AC84}) has an extension to a pseudo $n$-metric.
\begin{definition} \label{defhaus}
  Let $d:X^n \to \R$ be a pseudo $n$-metric on a set $X$.
  Then we define for nonempty subsets $A_j \subseteq X$, $j=1,\ldots,n$ the
  following quantities
  \begin{equation} \label{defdistset}
    \begin{aligned}
    \mathrm{dist}(x_1;A_2,\ldots,A_n) & = \inf\{d(x_1,x_2,\ldots,x_n):
    x_j \in A_j, j=2,\ldots,n \}, \\
    \mathrm{dist}(A_1;A_2,\ldots,A_n)& = \sup\{ \mathrm{dist}(x_1;A_2,\ldots,A_n)
    : x _1 \in A_1 \} , \\
    d_H(A_1,\ldots,A_n) & = \max_{j=1,\ldots,n} \mathrm{dist}(A_j;A_1,\ldots,
    A_{j-1},A_{j+1},\ldots,A_n) .
    \end{aligned}
  \end{equation}
\end{definition}
  For  an ordinary $2$-metric the construction \eqref{defdistset}
  leads to the familiar Hausdorff distance
  \begin{align*}
    d_H(A_1,A_2) & = \max(\mathrm{dist}(A_1;A_2),\mathrm{dist}(A_2;A_1))
  \end{align*}
  where $\mathrm{dist}(\cdot;\cdot)$ is the Hausdorff semidistance defined
  by
  \begin{align*}
    \mathrm{dist}(A_1;A_2) & = \sup _{x_1 \in A_1} \inf_{x_2 \in A_2} d(x_1,x_2).
  \end{align*}
  It is well known that  $d_H$  is a metric on the set of all closed subsets of $X$
  where 'closedness' is defined with respect to the given metric in $X$;
  see e.g. \cite[Ch.9.4]{AC84}.
  
It is not difficult to see that the map $d_H$ defined in \eqref{defdistset}
is  semidefinite and symmetric .
While the simplicial inequality is true for $n=2$, i.e. the Hausdorff distance,
we claim that it is generally false for  $n \ge 3$.
\begin{example} \label{ex5.1} (Counterexample for the simplicial inequality)\\
  Consider the case $n=3$ and four finite sets $A_1=\{w_1\}$,
  $A_2= \{x_1,\ldots,x_{N}\}$, $A_3=\{y_1,\ldots,y_{N}\}$,
  $A_4= \{z_1,\ldots,z_{N} \}$ where $N \ge 2$.
The set $X$ is defined as $X=A_1 \cup A_2 \cup A_3 \cup A_4$.
  For convenience we introduce the following abbreviations
  for $i=1$ and $j,k,\ell=1,\ldots,N$:
  \begin{align*}
    d_{0jk\ell}& = d(x_j,y_k,z_{\ell}),\quad d_{i0k\ell} = d(w_i,y_k,z_{\ell}),\\
    d_{ij0\ell}& = d(w_i,x_j,z_{\ell}),\quad d_{ijk0} = d(w_i,x_j,y_k).
    \end{align*}
  The values of the $3$-metric with three arguments from different sets $A_{\nu}$
  in $X$ are defined as follows:
  \begin{align*}
    d_{0jk\ell}&= \begin{cases} 0, & j =k =
      \ell , \\
      1, & \text{otherwise},
    \end{cases}\\
    d_{i0k\ell}&= \begin{cases} 1, & k  = \ell , \\
      0, & \text{otherwise},
    \end{cases}\\
    d_{ij0\ell}& = \begin{cases} 1, & j  = \ell  ,\\
      0, & \text{otherwise},
    \end{cases}\\
    d_{ijk0}& = 1.
  \end{align*}
  Our first observation is
  \begin{equation*} \label{eq5.1}
    1= \min_{\ell=1,\ldots,N}(d_{ij0\ell}+d_{i0k\ell}+ d_{0jk\ell}),
  \end{equation*}
  which implies $d_{ijk0}\le d_{ij0\ell}+d_{i0k\ell}+ d_{0jk\ell}=:S_{\ell}$
  for all $i,j, k, \ell$.
  In fact, we find
  \begin{align*}
    S_{\ell}= \begin{cases} 1+1 + 0, & k =j =
      \ell , \\
      0+0 + 1, &
      k =j \neq \ell ,\\
      1+0 + 1, & k \neq j = \ell ,\\
      0+1 + 1, &  j \neq \ell 
      =k ,\\
      0+0+1, & k \neq j \neq \ell 
      \neq k .
    \end{cases}
  \end{align*}
  Further, the simplicial inequalities
  \begin{equation} \label{eq5.2}
  \begin{aligned}
    d_{ij0\ell}&  \le d_{ijk0}+d_{i0k\ell} + d_{0jk\ell},\\
    d_{i0k\ell}& \le d_{ijk0}+ d_{ij0\ell}+ d_{0jk\ell}. \\
    d_{0jk\ell}& \le d_{ijk0}+d_{ij0\ell}+ d_{i0k\ell},
  \end{aligned}
  \end{equation}
  hold, because $d_{ijk0}=1$ and all metric values lie in $\{0,1\}$.

  By the definition of $d_{ijk\ell}$ and  $N\ge2$ we also obtain
  \begin{align*}
     \forall i \quad \exists j,\ell & : d_{ij0\ell}=0;\quad
    \forall j \quad \exists i,\ell : d_{ij0\ell}=0; \quad
    \forall \ell \quad \exists i,j  : d_{ij0\ell}=0;\\
       \forall i \quad \exists k,\ell & : d_{i0k\ell}=0; \quad
    \forall k \quad \exists i,\ell  : d_{i0k\ell}=0; \quad
    \forall \ell \quad \exists i,k  : d_{i0k\ell}=0;\\
    \forall j \quad \exists k,\ell& : d_{0jk\ell}=0; \quad
    \forall k \quad \exists j,\ell : d_{0jk\ell}=0; \quad
    \forall \ell \quad \exists j, k  : d_{0jk\ell}=0.
  \end{align*}
  Following the definition of $d_H$ in \eqref{defdistset} this
  implies
  \begin{align*}
    d_H(A_1,A_2,A_4)& = \max\big(\max_i \min_{j,\ell}d_{ij0\ell},
    \max_j \min_{i,\ell}d_{ij0\ell},\max_\ell \min_{i,j}d_{ij0\ell}\big)=0,\\
     d_H(A_1,A_3,A_4)& = \max\big(\max_i \min_{k,\ell}d_{i0k\ell},
     \max_k \min_{i,\ell}d_{i0k\ell},\max_\ell \min_{i,k}d_{i0k\ell}\big)=0,\\
      d_H(A_2,A_3,A_4)& = \max\big(\max_j \min_{k,\ell}d_{0jk\ell},
    \max_k \min_{j,\ell}d_{0jk\ell},\max_\ell \min_{j,k}d_{0jk\ell}\big)=0.
  \end{align*}
      Note that we also have
     \begin{align*}
       1=\max\big(  \max_i(\min_{j,k} d_{ijk0}),\max_j(\min_{i,k} d_{ijk0}),
       \max_k(\min_{i,j} d_{ijk0})\big)=d_H(A_1,A_2,A_3).
     \end{align*}
     This contradicts 
     the simplicial inequality \eqref{tri}.

     In the following we define $d$ for all triples where at least two elements
     lie in the same set, so that the axioms of a $3$-metric are satisfied.
     Without loss of generality we  define $d(\xi_1,\xi_2,\xi_3)$ when $\xi_1,\xi_2,\xi_3$
     lie in sets $A_1$, $A_2$, $A_3$, $A_4$ with increasing lower index.
     All other values are then given by the permutation symmetry \eqref{perm}.
     Further, we set $d(\xi_1,\xi_2,\xi_3)=0$ if two of the arguments
     agree (semidefiniteness)  or if all three arguments $\xi_1,\xi_2,\xi_3$ lie in the
     same set $A_\nu$. It remains to define the following quantities
     \begin{align*}
       d_{i,j \iota,00}=d(w_i,x_j,x_{\iota}) & = 0, \\
       d_{i0,k \kappa,0}=d(w_i,y_k,y_{\kappa}) & =0 ,\\
       d_{i00,\ell \lambda}=d(w_i,z_{\ell},z_{\lambda}) & = 0,\\
       d_{0,j \iota ,0 \ell}=d(x_j,x_{\iota},z_{\ell}) & = \begin{cases}
         1, & (\iota = \ell) \vee (j = \ell), \\
         0, & \text{otherwise}, \end{cases} \\
       d_{0j0,\ell \lambda}= d(x_j,z_{\ell},z_{\lambda}) & =
        \begin{cases}
         1, & (j = \ell) \vee (j = \lambda), \\
         0, & \text{otherwise}, \end{cases}\\
        d_{00,k \kappa,\ell}=d(y_k,y_{\kappa},z_{\ell}) & = 
        \begin{cases} 1, & (k = \ell) \vee (\kappa = \ell), \\
          0, & \text{otherwise},
        \end{cases}\\
        d_{00k,\ell \lambda}=d(y_k,z_{\ell},z_{\lambda}) & = \begin{cases}
          1, & (k=\ell) \vee (k= \lambda),\\
          0,& \text{otherwise}, \end{cases}\\
        d_{0,j \iota,k0}=d(x_j,x_{\iota},y_k) & =1 , \\
        d_{0j,k \kappa,0}= d(x_j,y_k,y_{\kappa})& = 1.
     \end{align*}
     We have to check all simplicial inequalities involving
     two arguments from the same set and for which the $d$-value
     on the lefthand side is $1$. We use the symbol $*$ to denote values which
     are either $0$ or $1$ and $(* + 1)$ to denote a sum which is $\ge 1$:
\begin{align*}
  1 & = d_{0jk\ell}, \quad \text{if}\; (j \neq k) \vee (k \neq \ell),\\
  & \le d_{0,j \iota,k0}+d_{0\iota k \ell}+ d_{0,j \iota,0\ell}=1+*+*,\\
  & \le d_{0j,k \kappa,0}+d_{0j \kappa \ell}+d_{00,k \kappa,\ell}= 1+* +*,\\
  & \le d_{0j0, \ell \lambda}+d_{00k,\ell \lambda}+d_{0jk \lambda}=* + (* + 1).
\end{align*}
For the last line note that either $ d_{0jk \lambda}=1$ or ($ d_{0jk \lambda}=0$,
$j=k=\lambda$, $d_{00k,\ell \lambda}=1$).
   \begin{align*}
  1 & = d_{ij0 \ell}, \quad \text{if}\; (j =\ell),\\
  & \le d_{i,j \iota,00}+d_{i\iota 0 \ell}+ d_{0,j \iota,0\ell}=*+*+1,\\
   & \le d_{0j0, \ell \lambda}+d_{i00,\ell \lambda}+d_{ij \lambda}=1 + * + *.
  \end{align*}
   \begin{align*}
  1 & = d_{i0k\ell}, \quad \text{if}\; (k =\ell),\\
  & \le d_{i,0,k \kappa,0}+d_{i 0 \kappa \ell}+ d_{00,k \kappa,\ell}=*+*+1,\\
   & \le d_{i00, \ell \lambda}+d_{i0k \lambda}+d_{00k,\ell \lambda}=*+* + 1 .
   \end{align*}
   \begin{align*}
  1 & = d_{ijk0},\\
  & \le d_{i,j \iota,00}+d_{i\iota k 0}+ d_{0,j \iota,k0}=*+1+*,\\
  & \le d_{i0,k \kappa,0}+d_{ij \kappa 0}+d_{0j,k \kappa,0}= *+1 +*.
   \end{align*}
   In each of these three cases we  test only two inequalities since
   the remaining three are satisfied by \eqref{eq5.2}.
   
   It remains to discuss $d$-values where either $j\iota$, $k \kappa$ or
   $\ell \lambda$ appears on the lefthand side.
   \begin{align*}
     1& = d_{0,j \iota, k0}, \\
     & \le d_{ijk0} + d_{i \iota k 0} + d_{i,j \iota,00}= 1 + 1 + *, \\
     & \le d_{0j,k \kappa, 0}+ d_{0 \iota,k \kappa, 0}+ d_{0,j \iota, \kappa,0}
     = * + * + 1, \\
     &\le d_{0jk \ell} + d_{0,j \iota,0 \ell} + d_{0 \iota k \ell} = (1+ *) + *.
   \end{align*}
   For the last line note that either $d_{0jk \ell}=1$ or ( $d_{0jk \ell}=0$,
   $j=k=\ell$, $d_{0, j \iota,0 \ell}=1$).
    \begin{align*}
     1& = d_{0j ,k \kappa, 0}, \\
     & \le d_{ijk0} + d_{i j \kappa 0} + d_{i0, k \kappa,0 }= 1 + 1 + *, \\
     & \le d_{0,j \iota,k 0}+ d_{0,j \iota, \kappa, 0}+ d_{0\iota,k \kappa,0}
     = * + * + 1, \\
     &\le d_{00,k \kappa, \ell}+ d_{0j k \ell}  + d_{0 j \kappa \ell} = (1+*) +  *.
    \end{align*}
    For the last line note that either $d_{0jk \ell}=1$ or ($d_{0jk \ell}=0$,
    $j=k=\ell$, $d_{00,k \kappa, \ell}=1$).

For the final four cases the indices of the metric value on the left
satisfy an extra condition
   \begin{align*}
     1& = d_{0,j \iota, 0 \ell}, \quad \text{if}\; (j= \ell) \vee (\iota = \ell),  \\
     & \le d_{ij0 \ell} + d_{i \iota 0 \ell} + d_{i,j \iota,00}= (1 + *) + *, \\
     & \le d_{0,j \iota,k 0}+ d_{0 j k \ell}+ d_{0 \iota k \ell }
     = 1 + * + *, \\
     &\le d_{0j0, \ell \lambda} + d_{0\iota0, \ell \lambda} + d_{0,j \iota, 0 \lambda}
     = (1+ *) + *.
   \end{align*}
   For last line note that $d_{0j0, \ell \lambda}=1$ if $j = \ell$ and
   $d_{0\iota0, \ell \lambda}=1$ if $\iota=\ell$.
    \begin{align*}
      1& = d_{00,k \kappa, \ell},  \quad \text{if} \; (k=\ell)\vee (\kappa=\ell),\\
     & \le d_{i0k\ell} + d_{i 0 \kappa \ell} + d_{i0, k \kappa,0 }= (1 + *) + *, \\
     & \le d_{0,j k \ell}+ d_{0 j  \kappa \ell}+ d_{0j,k \kappa,0}
     = * + * + 1, \\
     &\le d_{00k , \ell \lambda}+ d_{00 \kappa, \ell \lambda}  + d_{0 , k \kappa, \lambda} = (1+*) +  *.
    \end{align*}
    For the last line note that $d_{00k , \ell \lambda}=1$ if $k = \ell$ and
    $d_{00 \kappa, \ell \lambda}=1$ if $\kappa=\ell$.
     \begin{align*}
     1& = d_{0j 0, \ell \lambda}, \quad \text{if}\; (j= \ell) \vee (j = \lambda),  \\
     & \le d_{ij0 \ell} + d_{i j 0 \lambda} + d_{i00,\ell \lambda}= (1 + *) + *, \\
     & \le d_{0,j \iota,0 \ell}+ d_{0 ,j \iota,0 \lambda}+ d_{0 \iota 0, \ell  \lambda }
     = (1 + *) + *, \\
     &\le d_{0jk \ell} + d_{0jk\lambda} + d_{00k, \ell \lambda}
     = *+ (1+ *).
     \end{align*}
     For the last line note that  either $d_{0jk \lambda}=1$ or ($d_{0jk \lambda}=0$,
     $j=k=\lambda$, $d_{00k, \ell \lambda}=1$).
     \begin{align*}
     1& = d_{00k, \ell \lambda}, \quad \text{if}\; (k= \ell) \vee (k = \lambda),  \\
     & \le d_{i0k \ell} + d_{i 0 k \lambda} + d_{i00,\ell \lambda}= (1 + *) + *, \\
     & \le d_{0j k \ell}+ d_{0 j k \lambda}+ d_{0 j 0, \ell  \lambda }
     =  * + (1+*), \\
     &\le d_{00,k \kappa, \ell} + d_{00,k \kappa, \lambda} + d_{00 \kappa, \ell \lambda} 
     = (1+*)+ *.
     \end{align*}
     For the last line note that $d_{00,k \kappa, \ell}=1$ if $k = \ell$ and
     $d_{00,k \kappa, \lambda}=1$ if $k = \lambda$.

     Finally, in case $N \ge 3$ we have to test the last six blocks of
     simplicial inequalities when a double pair $j \iota$ resp. $k \kappa$
     resp. $\ell \lambda$ is prolonged by another index $j'$ resp.
     $k'$ resp. $\ell'$. For example, two characteristic cases are the following:
     \begin{align*}
       1& = d_{0,j \iota, k0}, \\
       & \le d_{0,j \iota j',00}+d_{0,j j',k0}+ d_{0, \iota j',k0}=0+1+1, \\
       1& = d_{0,j \iota, 0 \ell}, \quad \text{if}\; (j= \ell) \vee (\iota = \ell),  \\   & \le d_{0,j \iota j',00}+ d_{0,j j',0 \ell}+d_{0, \iota j', 0 \ell}=0 +(1+ *).
     \end{align*}
     For the last line note that $d_{0,j j',0 \ell}=1$ if $j=\ell$ and
$d_{0, \iota j', 0 \ell}=1$ if $\iota=\ell$.
     By inspection one finds that the remaining four cases can be handled
     in the same way.
\end{example}
While this example shows that the definition \eqref{defdistset} with an
arbitrary pseudo $n$-metric does not necessarily define a pseudo $n$-metric
on subsets, it is still possible that some of the special $n$-metrics from
Section \ref{s:lin} will have this property.


\end{document}